\documentclass[11pt]{amsart}

\usepackage[margin=1in]{geometry}
\usepackage{amssymb,amsfonts,amsthm,mathtools,
enumerate,verbatim,color,tikz,url}
\usetikzlibrary{arrows}
\usepackage{verbatim}
\usepackage{hyperref}
\theoremstyle{plain}
\newtheorem{theorem}{Theorem}[section]

\theoremstyle{definition}

\newcommand\tb{\textbf}
\newcommand{\N}{\mathbb{N}}

\newcommand{\mcal}{\mathcal}
\DeclareMathOperator{\dom}{\mathrm{dom}}

\newcommand{\ha}{{\sf HA}}                         
\newcommand{\pa}{\ensuremath{\mathsf{PA}}}         

\newcommand{\aca}{\ensuremath{\mathsf{ACA}_0}}

\newcommand{\haw}{\mathsf {HA}^\omega}                
\newcommand{\ehaw}{{\sf E}\text{-}{\sf HA}^\omega}      
\newcommand{\ehawo}{{\widehat{{\sf E}\text{-}{\sfHA}}}
\newcommand{\wehaw}{\mathsf{WE}\text{-}\mathsf {HA}^\omega}          
\strut^\omega_{\raise2pt\hbox{\scriptsize${\mathord{\upharpoonright}}$}}}  
\newcommand{\whawo}{{\widehat{\mathsf{WE}\text{-}\mathsf{HA}}}        
	\strut^\omega_{\raise2pt\hbox{\scriptsize${\mathord{\upharpoonright}}$}}}
\newcommand{\epawo}{{\widehat{{\sf E}\text{-}                
			{\sf PA}}}\strut^\omega_{\raise2pt\hbox{\scriptsize${\mathord{\upharpoonright}}$}}}   
\newcommand{\epato}{{\widehat{{\sf E}\text{-}{\sf PA}}}\strut^2_{\raise2pt\hbox{\scriptsize${\mathord{\upharpoonright}}$}}}  
\newcommand{\wepawo}{{\widehat{\mathsf{WE}\text{-}\mathsf{PA}}}        
	\strut^\omega_{\raise2pt\hbox{\scriptsize${\mathord{\upharpoonright}}$}}}
\newcommand{\wepato}{{\widehat{\mathsf{WE}\text{-}\mathsf{PA}}}        
	\strut^2_{\raise2pt\hbox{\scriptsize${\mathord{\upharpoonright}}$}}}
\newcommand{\ac}{{\sf AC}}
\newcommand{\dc}{{\sf DC}}
\newcommand{\re}{\ \mathrm{realizes} }
\newcommand{\fo}{\Vdash }

\newcommand{\imp}{\rightarrow}

\newcommand{\biimp}{\leftrightarrow}

\newcommand{\pair}[2]{\langle #1,#2\rangle}
\newcommand{\kl}[1]{\{#1\}}

\newcommand{\vp}{\varphi}

\newcommand{\ind}{\Lambda} 

\title{On Goodman realizability}
\author[Frittaion]{Emanuele Frittaion}
\address{Departamento de Matem\'{a}tica,  Universidade de Lisboa, Portugal}

\begin{document}


\maketitle
\tableofcontents

\begin{abstract}
Goodman's theorem (1976) states that $\haw+\ac+{\sf DC}$ is conservative over $\ha$. The same result applies to the extensional case, that is, $\ehaw+\ac+{\sf DC}$ is also conservative over $\ha$.  This is due to Beeson (1979).   In this paper we modify Goodman realizability \cite{goodman} and provide a new proof of the extensional case. 
\end{abstract}

\section{Introduction}
Goodman's theorem states that intuitionistic finite-type arithmetic $\haw$ plus the axiom of choice $\ac$ plus the axiom of (relativized) dependent choice $\sf RDC$ is conservative over $\ha$. This result does not extend to the classical case. In fact, adding classical logic leads to a system as strong as full second-order arithmetic. On the other hand, the same result does not apply to subsystems of $\haw$ with restricted induction  \cite{kohl99}. The original proof of this theorem is due to Goodman \cite{goodman76} and is based on his arithmetic theory of constructions.   A direct proof appears in  Goodman  \cite{goodman} and relies on a realizability notion  which blends together  the model $\sf HRO$ of the hereditarily recursive operations, Kleene recursive realizability, and Kripke semantics. A restyling of Goodman's proof can be found in Beeson \cite{beeson}, where  Goodman realizability \cite{goodman} splits into a realizability part and  a forcing part.  The same framework appears in other papers on Goodman's theorem \cite{coq,lava}.

In this paper we introduce a customized version of Goodman realizability and give a new proof of the  extensional case, that is, we show that $\ehaw+\ac+\dc$ is conservative over $\ha$. The only known proof of this result is also due to Beeson \cite{beeson} (\footnote{Apparently Beeson does not consider the axiom of relativized dependent choice.}). 
Beeson \cite[page 9]{beeson} points out that if we combine Kreisel's modified realizability with the model  $\sf HEO$ of the  hereditarily effective operations then we obtain a realizability interpretation of $\ehaw+\ac$ into $\ha$ but we cannot force first-order formulas to be self-realizing. The problem is that in Kreisel's modified realizability the realizers are total object.  Beeson's solution is to consider a version of Kreisel's modified realizability for finite-type partial functionals. We propose a different solution (\footnote{Actually, our attempt to follow Beeson's strategy turned out to be amiss. The problem is the soundness. To make a long story short, partial functionals with extensional equality only with respect to total functionals do not behave well.}). We follow Goodman's blueprint and define a realizability notion which  combines the model $\sf HEO$ of the hereditarily effective operations  with an \emph{extensional} version of Kleene recursive realizability. Indeed, we define an \emph{extensional} equality between realizers of a formula $\vp$, denoted $p\fo (a,b)\colon \vp$, where $a$ and $b$ are natural numbers and $p$ is a partial function from natural numbers to natural numbers,  whose intended meaning is \lq$p$ forces $a$ and $b$ to be equal realizers of $\vp$\rq. By letting  $p\fo a\colon \vp$ be a shorthand for $p\fo (a,a)\colon \vp$, we therefore prove:
\begin{itemize}
	\item Soundness: for every sentence $\vp$ provable in $\ehaw+\ac+\dc$ there exists an index $a$ such that for every definable set $T$ of partial functions $\ha\vdash \forall p\in T\ (p\fo a\colon \vp)$
	\item Self-realizability: for every first-order sentence $\vp$ there exists a non-empty definable set $T$ of finite partial functions such that $\ha\vdash \forall p\in T\ \forall a\ ((p\fo a\colon \vp)\imp \vp)$.
\end{itemize}
Goodman  writes $\pair pa\fo \vp$. Our choice of notation $p\fo a\colon \vp$ arises from viewing $\vp$ as a type, the type of its realizers. 
It is worth  noticing that the indices in the proofs of both the soundness and the self-realizability theorems are exactly the ones in \cite[Theorems 1 and 4]{goodman}. The verification is straightforward, although burdensome. \\

The outline  of this paper is as follows. In Section \ref{arithmetic} we describe the system $\ehaw+\ac+\dc$.   In Section \ref{goodman} we review Goodman realizability.
In Section \ref{extensional} we present our version of Goodman realizability and prove the extensional case. Finally, we show in Section \ref{realizability} how to combine the model $\sf HEO$ with Kleene recursive realizability to obtain a realizability interpretation of  $\ehaw+\ac+\dc$ into $\ha$,  where the realizers are, as in Kreisel's modified realizability, total objects.  Actually,  one can define a Goodman version of this interpretation and prove its soundness in $\ha$. What fails in $\ha$, but not in $\pa$, is the self-realizability theorem, which makes this approach unsuitable to prove the conservativity of  $\ehaw+\ac+\dc$ over $\ha$. We discuss this further at the end of Section \ref{realizability}.

\section{Extensional finite-type arithmetic with choice}\label{arithmetic}

We assume familiarity with systems of arithmetic in all finite types (see, e.g.,\cite{kohl,troelstra}). 
Here, we formalize $\ehaw+\ac+\dc$ following Goodman \cite{goodman}. In particular, we have equality at every finite type.  We denote types by $A,B,C,\ldots$  $N$ denotes the type of natural numbers. As in Goodman, we also have product types $A\times B$.
\[   \text{Types } ::= N \mid  A\times B \mid A\to B \]
The language $\mcal L$ of finite-type arithmetic includes variables at all sensible types, the constant $0$ of type $N$, the constant $S$ for successor of type $N\to N$, functions symbols for all primitive recursive functions, constants for combinators $\varPi$, $\varSigma$,  recursors $R$, pairing $D$ and projections $D_0, D_1$ at all sensible types.
\[  \text{ Terms } ::= \text{variables} \mid \text{constants} \mid f(\alpha_1,\ldots,\alpha_k)\mid \alpha\beta \]
Here, $f$ is a function symbol for a primitive recursive function and $\alpha_i$ are terms of type $N$. In the application term $\alpha\beta$, $\alpha$ is a term of type $A\to B$ and $\beta$ is a term of type $A$. To be precise, for all types $A$ and $B$, we have  a binary function symbol $Ap$ of sort $(A\to B)\times A\to B$ . Therefore,
$\alpha\beta$ stands for $Ap(\alpha,\beta)$ (\footnote{Note that we are dealing with a many-sorted first-order theory. The types are indeed the sorts of the language.}).
We usually indicate the type of a term by a superscript as in $\alpha^A$.
For every type $A$ we have a binary relation symbol $=_A$ for equality of type $A$. 
\[  \text{ Formulas } ::= \alpha=_A\beta \mid \vp\land \psi\mid \vp\lor \psi \mid \vp\imp \psi \mid \exists x^A\vp  \mid \forall x^A\vp \]
Here, $\alpha$ and $\beta$ are terms of type $A$. As usual in systems of arithmetic, we write  $\neg\vp$ as an abbreviation for $\vp \imp 0=S0$. \\

Axioms and rules: 
\begin{enumerate}
	\item[] Propositional logic:
	\item[(1)] $\vp\imp\vp$
	\item[(2)] Modus ponens: if $\vp$ and $\vp\imp\psi$ then $\psi$
	\item[(3)] Syllogism: if $\vp\imp\psi$ and $\psi\imp\chi$ then $\vp\imp\chi$
	\item[(4)] $\vp\land\psi\imp\vp$, $\vp\land\psi\imp\psi$
	\item[(5)] $\vp\imp\vp\lor\psi$, $\psi\imp\vp\lor\psi$
	\item[(6)] If $\chi\imp\vp$ and $\chi\imp\psi$ then $\chi\imp\vp\land\psi$
	\item[(7)] If $\vp\imp\chi$ and $\psi\imp\chi$ then $\vp\lor\psi\imp\chi$
	\item[(8)] $\vp\land\psi\imp\chi$ iff $\vp\imp(\psi\imp\chi)$
	\item[(9)] $0=S0\imp\vp$
	\item[] Predicate logic:
	\item[(10)] If $\vp\imp\psi$ then $\vp\imp\forall x\psi$, for $x$ not free in $\vp$
	\item[(11)] If $\vp\imp\psi$ then $\exists x\vp\imp\psi$, for $x$ not free in $\psi$
	\item[(12)] $\forall x\vp\imp\vp(\alpha)$, $\vp(\alpha)\imp\exists x\vp$, for any term $\alpha$ free for $x$ in $\vp$
	\item[] Axioms for constants:
	\item[(13)] $\neg(0=Sx)$, $Sx=Sy\imp x=y$, for  $x,y$ variables of type $N$
	\item[(14)] Defining equations of primitive recursive functions 
	\item[(15)] Combinators: $\varPi xy=x$ and $\varSigma xyz=(xz)(yz)$ at all sensible types
	\item[(16)] Recursors: $Rxy0=x$ and $Rxy(Sz)=y(Rxyz)z$ at all sensible types
	\item[(17)] $D_0(D xy)=x$ and $D_1(Dxy)=y$, for $x$ of type $A$ and $y$ of type $B$
	\item[(18)] $x=D(D_0x)(D_1x)$, for  $x$ of type $A\times B$
	\item[] Induction:
	\item[(19)]  If $\vp(0)$ and $\vp(x)\imp\vp(Sx)$ then $\vp(x)$, for $x$ variable of type $N$
	\item[] Equality:
	\item[(20)] $x=x$, for $x$ of type $A$
	\item[(21)] $x=y\lor \neg (x=y)$, for $x$ and $y$ of type $N$
	\item[(22)]  Leibniz: $x= y\land \varphi(x)\imp \varphi (y)$, for $x$ and $y$ of type $A$
	\item[] Extensionality:
	\item[(23)] $\forall z(xz=yz)\imp x=y$, for $z$ of type $A$ and $x,y$ of type $A\to B$
	\item[] Axiom of choice:
	\item[(24)] $\forall x\exists y\varphi(x,y)\imp \exists z\forall x\varphi(x,zx)$, for $x$  of type $A$, $y$  of type $B$, and $z$  of type $A\to B$
	\item[] Axiom of relativized dependent choice:
	\item[(25)]  $\forall x[\vp(x)\imp \exists y(\vp(y)\land \psi(x,y))]\imp \forall x[\vp (x)\imp \exists z(z0=x\land \forall v\psi(zv,z(Sv)))]$, for $x$ and $y$  of type $A$, $z$  of type $N\to A$ and $v$  of type $N$. 
\end{enumerate} 
\smallskip 

$\ehaw$ is the system (1)--(23). Indeed, axiom $\neg(0=Sx)$ in (13) and axiom (21) are provable from the other axioms.  Similarly,  the axioms (14) are redundant (combinators and recursors enable us to define every primitive recursive function). The reason to include them is to see $\ha$ as a subsystem of $\ehaw$ (once we identify the function symbol $S$ and the type $N\to N$ constant $S$). Indeed, let us define $\ha$ to be the restriction of $\ehaw$ to first-order formulas. By a first-order  (arithmetical) formula we mean a formula of $\mathcal L$ which contains only first-order terms and quantifiers of type $N$, where a first-order term $\alpha$ is defined as follows:
\[   \alpha ::= x^N \mid 0 \mid S\alpha \mid f(\alpha_1,\ldots,\alpha_k) \]
In particular, in $\ha$ we do not have  axioms (15)--(18) and (23)--(24).\\   

Goodman \cite{goodman} defines $\haw$ without (axioms for) combinators and recursors. On the other hand, Goodman has decidable equality at every type, that is, axiom (21) holds with no restriction.   Goodman's formulation of $\haw$ is thus different from other versions present in the literature. Our version of $\ehaw$ has equality at all types and is equivalent to a conservative extension of Kohlenbach's $\ehaw$ \cite{kohl} (Kohlenbach's system does not have product types). Note that the restriction of axiom (21) to type $N$ is necessary, for otherwise we would have excluded middle for all formulas (by proving that for every formula $\vp$ with free variables $x_1,\ldots,x_k$ there exist closed terms $\alpha$ and $\beta$ such that $\vp(x_1,\ldots,x_k) \biimp \alpha x_1\ldots x_k=_A \beta x_1\ldots x_k$, where the type $A$ depends on $\vp$). 

\subsection{Notation}
As usual, we denote by $\bar n$ the numeral corresponding to the natural number $n$.  We omit the overline notation when clear from context.

Kleene application. Let $\kl a(n)$ be Kleene bracket application. As usual, we think of $a$ as a code for a partial recursive function from $\N$ to $\N$. Similarly, let $\kl a^p(n)$ be Kleene bracket application relative to the oracle $p$. Throughout the paper $p$ is a partial function from $\N$ to $\N$. To further ease notation,  we write, e.g.,  $ab$ for $\kl a(b)$ and $a^pb$ for $\kl a^p(b)$. More in general, we write $an_1\ldots n_k$ for $(an_1\ldots n_{k-1})n_k$. Similarly for $a^pn_1\ldots n_k$. For instance, $a^pbc$ stands for $(a^pb)^pc$.  We might write $a^p(b,c)$ in the interest of readability.

We tacitly assume  the s-m-n theorem and the recursion theorem in $\ha$ (\footnote{ s-m-n theorem: For all standard $m,n$ there is a primitive recursive function $s$ such that $\ha$ proves \[\kl a(x_1,\ldots,x_m,y_1,\ldots,y_n)\simeq \kl {s(a,x_1,\ldots,x_n)}(y_1,\ldots,y_m). \]
	Recursion theorem: $\ha$ proves that for all $a,n$ there is $e$ such that \[ \kl a(e,x_1,\ldots,x_n)\simeq \kl e(x_1,\ldots, x_n).\]}).

Kleene equality. 
Write $an=bm$ if $an$ and $bm$ are both defined and equal. Write $an\simeq bm$ iff either $an$ and $bm$ are both undefined or else $an=bm$.

By $\ind n.f(n)$ we denote a canonical index for the partial recursive function  $f$. In general, $\ind n_1n_2\ldots n_k.f(n_1,\ldots,n_k)$ stands for $\ind n_1.(\ldots (\ind n_k.f)\ldots)$.  

Types. To save on parentheses we agree that $\times$ binds stronger that $\to$. Thus $A\times B\to C$ stands for $(A\times B)\to C$. We use association to the right. Let $A_1\to A_2\to \ldots\to A_k$ be $A_1\to (A_2\to \ldots (A_{k-1}\to A_k)\ldots)$ and  $A_1\times\ldots \times A_k$ be 
$A_1\times (A_2\times \ldots (A_{k-1}\times A_k)\ldots )$. 

By $x_1,\ldots,x_k$ we denote a list of $k$ variables. Similarly, by $n_1,\ldots,n_k$ we denote a list of $k$ natural numbers. This includes the empty list. Say for $k=0$. 

Tuples. Let $\langle\cdot,\cdot\rangle$ be any primitive recursive bijection from pairs of natural numbers to natural numbers  with primitive recursive projections $(\cdot)_0$ and $(\cdot)_1$. Let $\langle n_1,\ldots,n_k\rangle$ be  $\langle n_1,\langle n_1,\ldots,n_k\rangle\rangle$. Note that a natural number $n$ codes a tuple of length $k$ for any given $k$. To ease notation, we assume a primitive recursive coding of sequences of natural numbers with primitive recursive projections $(\cdot)_i$ for any natural number $i$. We thus identify the tuple $\langle n_1,\ldots, n_k\rangle$ with the corresponding sequence of length $k$ and  write $(n)_i$ for the $i$-th component of the sequence. For example, if $n=\langle n_1,n_2,n_3\rangle$, we write $(n)_2$ for $n_2$ instead of $((n)_1)_0$.

\section{Goodman realizability: review}\label{goodman}

We review Goodman realizability \cite{goodman} for $\ha^\omega+\ac+\dc$.  Note that the formalization of $\sf HRO$ in $\ha$ gives an interpretation $|\vp|$ of $\haw$ into $\ha$. In particular, for every sentence $\vp$ provable in $\haw$ there exists a number $a$ such that  $\ha\vdash a\re\ |\vp|$, where $a\re\ \vp$ is Kleene recursive realizability (\footnote{Actually we should write $\ha\vdash \bar a\re\ \vp$. We omit such subtleties.}).
Therefore:
\[ \haw \vdash \vp\ \ \xrightarrow{\sf HRO}\ \    \ha\vdash |\vp|\ \  \xrightarrow{\text{Kleene}}\ \  \ha\vdash a\re\ |\vp|\]
Since $|\ac|$ and $|\dc|$ are Kleene realizable, provably in $\ha$, we have that:

\[    \haw+\ac+\dc \vdash \vp\ \  \xrightarrow{{\sf HRO}\ +\ \text{Kleene}}\ \  \ha \vdash a\re\ |\vp| \] 

Goodman's realizability bundles together the above interpretations and combines them with a Kripke forcing relation. Let us see some details.  \\

Let  $T$ be a set of partial functions from natural numbers to natural numbers. We define $p\fo a\colon \vp$, where $p\in T$, $a$ is a number and $\vp$ is a formula of $\mcal L$. We tacitly quantify $p,q,r$ over $T$. Warning: the discussion is informal. See for instance Kleene recursive realizability vs its formalized version in $\ha$.  Note that in $\ha$ we cannot even talk about arbitrary partial functions. At the end of the section we state the formalized versions from which Goodman's theorem follows.   \\

(1) First, one defines $p\fo a\in A$, where $a$ is a number, $p\in T$ and $A$ is a type.  We only mention the clause for the arrow type.
\begin{itemize}
	\item $p\fo a\in A\to B$ if for every $q\supseteq p$ and for every number $n$ such that $q\fo n\in A$, there exists $r\supseteq q$ such that $a^rn$ is defined and $r\fo a^rn\in B$.
\end{itemize}
This is a (relativized) version of $\sf HRO$. The forcing relation is monotone: if $p\fo a\in A$ and $q\supseteq p$ then $q\fo a\in A$.

(2) Second, for any $p\in T$, let $\mathcal L_p$ be the language $\mcal L$ of $\ha^\omega$ augmented with constants $F^A_a$ of type $A$, for every number $a$ and every type $A$  such that $p\fo a\in A$. Note that $\mcal L_p\subseteq \mcal L_q$ if $p\subseteq q$. Then one defines the interpretation (value) $|\alpha|_p$ of a closed term $\alpha$ of $\mcal L_p$. This can be undefined. We only mention the clause for an application term $\alpha\beta$.
\begin{itemize}
	\item $|\alpha\beta|_p\simeq \kl {|\alpha|_p}^p(|\beta|_p)$
\end{itemize}
The value is monotone, that is, if $p\subseteq q$ and $|\alpha|_p$ is defined, then $|\alpha|_q$ is defined and $|\alpha|_p=|\alpha|_q$.

(3) Goodman realizability. Finally, one defines $p\fo a\colon \vp$, where $p\in T$, $a$ is a number, and $\vp$ is a sentence of $\mcal L_p$. 
\begin{itemize}
	\item $p\fo a\colon \alpha=\beta$ iff $|\alpha|_p$ and $|\beta|_p$ are both defined  and $|\alpha|_p=|\beta|_p$
	\item $p\fo a\colon \vp\land \psi$ iff $p\fo {(a)_0}\colon \vp$ and $p\fo  (a)_1\colon \psi$ 
	\item $p\fo a\colon \vp\lor\psi$ iff $(a)_0=0$ and $p\fo (a)_1\colon \vp$ or else $(a)_0=1$ and $p\fo {(a)_1}\colon \psi$
	\item $p\fo a\colon \vp\imp\psi$ if for every $q\supseteq p$ and for every number $b$ such that $q\fo b\colon \vp$, there exists $r\supseteq q$ such that $a^rb$ is defined and $r\fo a^rb\colon\psi$ 
	\item $p\fo a\colon \exists x^A\vp$ iff $p\fo (a)_0\in A$ and $p\fo {(a)_1}\colon \vp(F_{(a)_0})$
	\item $p\fo a\colon \forall x^A\vp$ iff for every $q\supseteq p$ and for every number $n$ such that $q\fo n\in A$, there exists $r\supseteq q$ such that $a^rn$ is defined and $r\fo {a^rn}\colon \vp(F_n^A)$ 	
\end{itemize} 

\begin{theorem}[Soundness of Goodman realizability, cf.\ {\cite[Theorem 1]{goodman}}]
	Suppose $\vp$ is a formula of $\mcal L$ such that $\haw+\ac+\dc\vdash \vp$ and let $T$ be a set of partial functions.  If the free variables of $\vp$ are among the distinct variables $x_1,\ldots,x_k$, then
	there exists an index $a$ such that $p\fo a\colon \forall x_1\ldots\forall x_k\vp$ for every $p\in T$.
	In particular, if $\vp$ is a sentence, then there is an index $a$ such that $p\fo a\colon \vp$ for every $p\in T$.   
\end{theorem}

\begin{theorem}[Self-realizability of first-order formulas, cf.\ {\cite[Theorem 4]{goodman}}]\label{goodman self}
	Suppose $\vp$ is a first-order formula of $\mcal L$ with free variables among the distinct variables $x_1,\ldots,x_k$.  Then there exist a nonempty set $T$ of finite partial functions and an index $a$ such that:
	\begin{itemize}
		\item if $\vp(\bar n_1,\ldots,\bar n_k)$ is true, then for every $p\in T$ there exists  $q\supseteq p$ such that  $a^qn_1\ldots n_k$ is defined and $q\fo a^qn_1\ldots n_k \colon \vp(\bar n_1,\ldots,\bar n_k)$.
		\item  if $p\fo b\colon \vp(\bar n_1,\ldots,\bar n_k)$, then  $\vp(\bar n_1,\ldots,\bar n_k)$ is true.
	\end{itemize} 
	In particular, if $\vp$ is a first-order sentence of $\mcal L$, then there exists a nonempty set $T$ of finite partial functions such that: 
	\begin{itemize}
		\item  if $\vp$ is true, then for every $p\in T$ there exists  $q\supseteq p$ and a number  $a$  such that   $q\fo a\colon \vp$
		\item if $p\fo b\colon\vp$, then  $\vp$ is true
	\end{itemize}   
\end{theorem}

In the self-realizability theorem we agree on letting $a^p$  be, say, $a^p0$ if the list $x_1,\ldots,x_k$ is empty. Alternatively, read $a^pn_1\ldots n_k$ as $a^p0n_1\ldots n_k$. 
\smallskip

Goodman's theorem follows by formalizing the above proofs in $\ha$. A set $T$ of finite partial functions is definable in $\ha$ if there exists a formula $\psi(x)$ of $\ha$  such that for any finite partial function $p$,  $p\in T$ iff $\psi(\bar p)$ is true. Here, we assume some fixed encoding of finite partial functions with natural numbers.

\begin{theorem}[Soundness in $\ha$, cf.\ {\cite[Theorem 5]{goodman}}]
	Suppose $\vp$ is a sentence of $\mcal L$ such that $\haw+\ac+\dc\vdash \vp$ and let $T$ be a definable set of  finite partial functions.  Then there exists an index $a$ such that 
	\[ \ha\vdash \forall p\in T\  (p\fo a\colon\vp). \]
\end{theorem}
\begin{theorem}[Self-realizability in $\ha$,  cf.\ {\cite[Theorem 6]{goodman}} ]
	Suppose $\vp$ is a first-order sentence of $\mcal L$. Then there exists a definable set $T$ of finite partial functions such that:
	\begin{itemize}
		\item $\ha\vdash \exists p\ (p\in T)$
		\item $\ha\vdash \vp\imp \forall p\in T\ \exists q\in T\ \exists a\ (q\supseteq p\land q\fo a\colon \vp)$
		\item $\ha\vdash \forall p\in T\ \forall b\ ((p\fo b\colon \vp)\imp \vp)$ 
	\end{itemize}
\end{theorem}
\begin{theorem}[Goodman's theorem, cf.\ {\cite[Theorem 7]{goodman}}]
	Suppose  $\vp$ is a first-order sentence of $\mcal L$. If $\haw+\ac+\dc\vdash \vp$ then $\ha\vdash \vp$.
\end{theorem}

\subsection{Remarks}

1. For  $T=\{p\}$,  we have $p\fo a\in A$ iff $a\in H_p^{A}$, where \[\mathsf{HRO}_p=\{H_p^{A}\colon A \text{ type}\}\]  is the model of hereditarily recursive operations relative to $p$.  In particular, $|\alpha|_p$ is the interpretation of $\alpha$ in ${\sf HRO}_p$.

2. For $T=\{p\}$,  Goodman realizability is essentially the $\sf HRO$ interpretation of $\haw$ into $\ha$  followed by Kleene recursive realizability relative to $p$.

3. The statement of \cite[Theorem 1]{goodman} is slightly different. Besides, it follows from the proof that we can actually choose $a$ in advance, that is, the index $a$ does not depend on $T$ but only on the proof of $\vp$. See our Theorem \ref{sound}.

4. The statement of \cite[Theorem 4]{goodman} is quite different from Theorem \ref{goodman self}. Moreover, Goodman's argument involves a rather complicated induction on the formula $\vp$.  A more straightforward argument is as in the proof of \cite[Lemma 4.1]{beeson}.  See also our proof of Theorem \ref{self}. 

5.  Goodman \cite[Theorem 5]{goodman} states that for every provable sentence $\vp$ and for every definable  set $T$ of finite partial functions there is an index $a$ such that
\[   \ha \vdash  \forall p\in T\ \exists q\in T\  (q\supseteq p\land q\fo a\colon\vp). \]
This is enough to prove Goodman's theorem. That our version of the soundness theorem  holds follows from the proof.

6. Goodman \cite[Theorem 6]{goodman} states that if
$\vp$ is a first-order sentence of $\haw$, then there exists a definable nonempty set $T$ of finite  partial functions and an index $a$  such that
\[ \ha\vdash \vp\imp \forall p\in T\ \exists q\in T\ (q\supseteq p\land q\fo a\colon \vp). \] 
However, this does not work in $\ha$. In fact, suppose $\vp$ is a sentence of the form $\psi_0\lor\psi_1$. Then we should be able to find an index  $a$ such that if $\vp$ is true then for all $p\in T$ there is $q\supseteq p$ such that $q\fo a\colon  \vp$.  Following the proof, we have an index $a_i$ for $\psi_i$ and $a$ should be either $\langle 0,a_0\rangle$ or $\langle 1,a_1\rangle$.  Of course we can choose neither beforehand. 

7. If we define Goodman realizability by replacing the clauses for $\imp$ and $\forall$ with the following:
\begin{itemize}
	\item $p\fo a\colon \vp\imp \psi$ iff for every $q\supseteq p$ and for every 
	number $b$ such that $q\fo b\colon \vp$, $a^qb$ is defined and $ q\fo {a^qb}\colon \psi$ 
	\item $p\fo a\colon \forall x^A\vp$ iff for every $q\supseteq p$ and for every number $n$ such that $q\fo n\in A$,  $a^qn$ is defined and $q\fo a^qn\colon \vp(F_n^A)$ 	
\end{itemize}
then we can still prove the soundness theorem but  we are no longer able to prove  the self-realizability theorem. This is where the  $\forall p\exists q\supseteq p$ definition is at work.

8. It clearly follows from Goodman's theorem that for every arithmetical sentence $\vp$ of $\haw$, if 
$\haw+\ac+\dc\vdash \vp$, then $\pa\vdash \vp$. Indeed, the proof of this weaker conservation result can be obtained by using  total functions instead of finite partial functions. Let  $\aca$ be the subsystem of second order-arithmetic with comprehension and induction restricted to  arithmetical formulas. It is well-known that $\aca$ is conservative over $\pa$.  One can show that if $T$ is a  set of total functions $p$ from $\N$ to $\N$ definable in the language of second-order arithmetic, and $\vp$ is  a provable sentence of $\ha^\omega+\ac+\dc$, then there is an index $a$ such that 
\[ \aca\vdash \forall p\in T\ p\fo a\colon \vp\]
Fix a first-order sentence  $\vp$ and let $T$ be the set of total functions from $\N$ to $\N$ which are Skolem functions for every subformula of $\vp$. Then $T$ is definable and  one can show that
\[   \aca\vdash \exists p\  p\in T\land \forall p\in T\forall b\ ((p\fo b\colon \vp) \imp \vp). \]

In fact, we can define a $p$ in $T$  by arithmetical comprehension.  By putting the things together, we have that if $\ha^\omega+\ac+\dc\vdash \vp$, then $\aca\vdash \vp$, and so $\pa\vdash \vp$, for any arithmetical sentence $\vp$.

Of course one needs to reprove the soundness and the self-realizability theorems and so this is not a real simplification. However, note that we can work with total functions. Goodman's proof is designed for $\ha$ and to obtain the conservation result one must use partial functions so that $\ha\vdash \exists p\ p\in T$. For instance we always have $\emptyset\in T$. 
With total functions the definition of Goodman realizability is  much simpler.  Basically, 
$p\fo a\colon \vp$ iff $a$ Kleene realizes $\vp$ relative to $p$. For instance,
\begin{itemize}
	\item $p\fo a\colon \vp\imp\psi$ iff for every $b$, if $p\fo b\colon \vp$ then $a^pb$ is defined and $p\fo a^pb\colon \psi$. 
\end{itemize}
In particular, the definition of $p \fo  a\colon \vp$ is local, that is, it does not depend on other $q\in T$. 

\subsection{Discussion}
One might wonder how Goodman's theorem  applies to concrete proofs in, say, $\haw+\ac$ of a first-order sentence $\vp$. How does the corresponding proof in $\ha$ look like? For instance, it is not difficult to see that there is a proof in $\haw+\ac$ of the collection principle
\[    \forall x<a \exists y \vp(x,y) \imp \exists b\forall x<a\exists y<b \vp(x,y), \]
where  $\vp(x,y)$ is any first-order formula, that uses only quantifier-free induction and countable choice (\footnote{From $\forall x < a \exists y \vp(x,y)$ obtain $\exists f \forall x < a \vp(x, fx)$ by countable choice. Then take $b = 1 + \max \{ fx | x < a \}$ so that we have $\forall x < a \exists y < b \vp(x,y)$.}).  In the proof of the soundness theorem no induction is used to realize $\ac$. However, the induction needed to realize quantifier-free induction depends on the complexity of the forcing conditions, which in this case depends on the complexity of $\vp(x,y)$. In other words, the amount of induction needed in the proof of collection in $\ha$ depends on the complexity of $\vp(x,y)$. 

It would be interesting to find a more direct proof of Goodman's conservation result that can be easily implemented so as to give a direct transformation of proofs in $\haw+\ac$ into proofs in $\ha$. Goodman's proof is not quite suitable for this purpose.  

\section{Goodman's theorem for choice and extensionality}\label{extensional}

We aim to define a tailored version of Goodman realizability for both choice and extensionality such that first-order formulas are self-realizing. As observed in Section \ref{goodman}, 
Goodman realizability  is essentially the $\sf HRO$ interpretation of $\haw$ into $\ha$ followed by Kleene recursive realizability. Note that the $\sf HRO$ interpretation of extensionality is not Kleene realizable (otherwise extensionality would hold in $\sf HRO$). Therefore Goodman realizability does not validate extensionality. Before presenting our version of Goodman realizability,  let us discuss some attempts to solve the problem. Let us focus on the soundness. In fact, adding the forcing relation, although crucial for the self-realizability theorem,  is a routine matter. \medskip

1. The most obvious attempt is to replace $\sf HRO$ with $\sf HEO$. This would validate extensionality.  The problem is that 
the $\sf HEO$ interpretation of $\ac$  is not Kleene realizable. For instance, (the interpretation of) 
\[  \forall x^{N\to N}\exists e^N\forall y^N( xy=\kl e(y)) \]
is  Kleene realizable (any index for the identity function will do) but
(the interpretation of) \[  \exists z^{(N\to N)\to N}\forall x^{N\to N}\forall y^N(xy=\kl {zx}(y)) \]
is not. Therefore, the following instance of the axiom of choice
\[   \forall x^{N\to N}\exists e^N\forall y^N( xy=\kl e(y))\imp \exists z^{(N\to N)\to N}\forall x^{N\to N}\forall y^N(xy=\kl {zx}(y)) \]
is not Kleene realizable.

2.  A second attempt is to consider  Kreisel's modified realizability followed by the $\sf HEO$ interpretation of $\ehaw$ into $\ha$.  That this cannot work follows from the fact that in Kreisel's modified realizability certain instances of independence of premise which are not provable in $\ha$ are nevertheless realizable. This is definitely a problem. 

3.  To avoid the problem in 2.\ one can first interpret $\ehaw$ into $\ha$ by using $\sf HEO$ and then reinterpret $\ha$ into itself by using a Kreisel's version of Kleene recursive realizability.  Note in fact that independence of premise is not Kleene realizable.  A direct realizability interpretation of $\ehaw+\ac+\dc$ into $\ha$ based on this idea can be found in Section \ref{realizability}. In this interpretation formulas $\vp$ come with a certain type $A$ and realizers of $\vp$ are objects of $\sf HEO$ of type $A$. By adding the forcing relation we obtain a realizability interpretation which is sound but fails to satisfy the self-realizability theorem. This is the problem described in the introduction.   \medskip

So much for the attempts. Our solution is  to  combine $\sf HEO$ with an extensional version of Kleene realizability. In a sense, the interpretation in 3.\ is also extensional since realizers are elements of $\sf HEO$. This is not the case here. We use $\sf HEO$ to interpret the quantifiers $\exists x^A$ and $\forall x^A$, but we do not see realizers as elements of $\sf HEO$. Actually, every formula is a type, the type of its realizers, and we have an extensional equality between realizers of the same formula.    The link is the following: if $p\fo a\colon \forall x^A\exists y^B\vp$ then $p\fo b\in A\to B$, where $b^pn\simeq (a^pn)_0$. This is what we need to realize $\ac$. On the other hand, extensionality is trivially realizable since it follows from its $\sf HEO$ interpretation.  Details follow.\\

We will repeatedly use the fact that  $p\subseteq q$ implies $ a^p \subseteq  a^q$, that is, for all $n$, if $a^pn$ is defined, then $a^qn$ is also defined and $a^pn=a^qn$.
Let $T$ be any set of partial functions from $\N$ to $\N$.  From now on, we tacitly quantify $p,q,r$ over $T$. As in Goodman, our discussion is informal. The reader will convince himself that all the definitions and proofs can be carried out in $\ha$. \\

Convention. For the sake of brevity  we write, e.g.,  $p\fo abc \in A$ for $a^pbc$ is defined and  $p\fo a^pbc\in A$. Similarly, 
we write, e.g., $p\fo (an,b)\colon \vp(F_{cm})$ for $a^pn$  and $c^pm$ are  defined and $p\fo (a^pn,b)\colon \vp(F_{c^pm})$. \\

(1) For every $p\in T$, for all numbers $a,b$ and for every type $A$, we define $p\fo a=_Ab$ and let $p\fo a\in A$ be $p\fo a=_Aa$. 

\begin{itemize}
	\item $p\fo a=_Nb$ iff $a=b$
	\item $p\fo a=_{A\times B} b$ iff $p\fo (a)_0=_A(b)_0$ and $p\fo (a)_1=_B(b)_1$
	\item $p\fo a=_{A\to B} b$ iff for all $q\supseteq p$ and for all $n,m$, if $q\fo n=_Am$ then there is $r\supseteq q$ such that $r\fo  an=_Bbm$.
\end{itemize}

This is a (relativized) version of $\sf HEO$. Properties. (E1) Monotonicity. If $p\fo a=_Ab$ and $q\supseteq p$ then $q\fo a=_Ab$. (E2) Reflexivity. $p\fo a=_Ab$ implies $p\fo a,b\in A$.  Symmetry. If $p\fo a=_Ab$ then $p\fo b=_Aa$. Transitivity. If $p\fo a=_Ab$ and $p\fo b=_Ac$ then $p\fo a=_Ac$. By induction on $A$ by using monotonicity.  (E3) $p\fo a=_{A\to B}b$ iff $p\fo a,b\in A\to B$ and for all $q\supseteq p$ and for all $n$, if $q\fo n\in A$ then there is $r\supseteq q$ such that $r\fo an=_Bbn$. (E4) For every type $A$ there is a number $0^A$ such that $p\fo 0^A\in A$ for all $p$. Let $0^N=0$, $0^{A\times B}=\langle 0^A,0^B\rangle$ and $0^{A\to B}$ be an index for the function $\lambda n. 0^B$. \\

(2) For any $p\in T$, let $\mcal L_p$ be the language $\mcal L$ of $\haw$ augmented with constants $F_a^A$ for any $a$ and $A$ such that $p\fo a\in A$.  We simply write $F_a$ when the type is clear from the context.  In particular, a term of $\mcal L$ is a term of $\mcal L_p$ for all $p$. Note that if $\alpha$ is a term of $\mcal L_p$ (we also say that $\alpha$ is a $p$-term) and $q\supseteq p$, then $\alpha$ is a $q$-term.    The value $|\alpha|_p$ of a closed term $\alpha$ of $\mcal L_p$ is defined as in Goodman by recursion on $A$. Note that $|\alpha|_p$ can be undefined. 
\begin{itemize}
	\item $|0|_p=0$
	\item $|S|_p=|S|$ is such that$|S|^pn=n+1$ for all $p\in T$
	\item $|g(\alpha_1,\ldots,\alpha_k)|_p\simeq g(|\alpha_1|_p,\ldots,|\alpha_k|_p)$, for $g$ primitive recursive and $\alpha_i$ of type $N$
	\item $|\varPi|_p=|\varPi|$ is such that  $|\varPi|^pab=a$ for all $p\in T$ 
	\item $|\varSigma|_p=|\varSigma|$ is such that $|\varSigma|^pabn\simeq a^pn(b^pn)$ for all $p\in T$
	\item $|R|_p=|R|$ is such that $|R|^pab0=a$ and  $|R|^pab(n+1)\simeq  b^p(|R|^pabn)n$  for all $p\in T$ 
	\item $|D|_p=|D|$ is such that $|D|^pab=\pair ab$ for all $p\in T$
	\item $|D_i|_p=|D_i|$ is such that $|D_i|^pa=(a)_i$  for all $p\in T$
	\item $|\alpha\beta|_p\simeq \kl{|\alpha|_p}^p(|\beta|_p)$
	\item $|F_a^A|_p= a$
\end{itemize}

$|R|$ exists by the recursion theorem. 
Properties. (V1) Monotonicity. If $|\alpha|_p$ is defined and $q\supseteq p$ then $|\alpha|_q$ is defined and $|\alpha|_p=|\alpha|_q$. (V2) For any term $\alpha$ with variables among the distinct variables $x_1,\ldots, x_k$ there exists an index $d$ such that $d^pn_1\ldots n_k\simeq |\alpha(F_{n_1},\ldots, F_{n_k})|_p$ for all $n_1,\ldots,n_k,p$. \\

Convention. For the sake of brevity, we write, e.g., $p\fo |\alpha|=_A|\beta|$ for $|\alpha|_p$ and $|\beta|_p$ are both defined and $p\fo  |\alpha|_p=_A|\beta|_p$. Similarly, we write $|\alpha|_p=|\beta|_p$ if $|\alpha|_p$ and $|\beta|_p$ are both defined and $|\alpha|_p=|\beta|_p$.\\

We now extend the definition of equality to terms.  Let $\alpha$ and $\beta$ be closed $p$-terms of type $A$. Then 
\begin{itemize}
	\item $p\fo \alpha=_A\beta$ iff for all $q\supseteq p$ there is $r\supseteq q$ such that $r\fo |\alpha|=_A|\beta|$.
\end{itemize}

From (V1) it follows that if $p\fo |\alpha|=_A|\beta|$ then $p\fo \alpha=_A\beta$.  \smallskip

Properties.  (T1) Monotonicity. If $p\fo \alpha=_A\beta$ and $q\supseteq p$ then $q\fo \alpha=_A\beta$. 
(T2) Reflexivity. If $p\fo \alpha=_A\beta$ then $p\fo \alpha,\beta\in A$.
Symmetry. If $p\fo \alpha=_A\beta$ then $p\fo \beta=_A\alpha$. Transitivity.  If $p\fo \alpha=_A\beta$ and $p\fo \beta=_A\gamma$ then $p\fo \alpha=_A\gamma$. 
(T3) If $\alpha$ is a closed $p$-term of type $A$ then $p\fo \alpha\in A$. 
(T4) Let $\gamma(x^A)$ be a $p$-term of type $B$, and $\alpha$ and $\beta$ be closed $p$-terms  of type $A$. If $p\fo \alpha=_A\beta$,  then $p\fo \gamma(\alpha)=_B\gamma(\beta)$. In particular, it follows from (T3) that if $|\alpha|_p=|\beta|_p$ then $p\fo \gamma(\alpha)=_B\gamma(\beta)$.  
(T5)  Let $\gamma(x^A),\delta(x^A)$ be $p$-terms of type $B$, and $\alpha,\beta$ be closed $p$-terms of type $A$. If $p\fo \gamma(\alpha)=_B\delta(\alpha)$ and $p\fo \alpha=_A\beta$, then  $p\fo \gamma(\beta)=_B\delta(\beta)$.   \\

\begin{proof}[Proof of (T3)]
	If $\alpha$ is a constant of $\mcal L$ of type $A$, then $|\alpha|_p=|\alpha|$ for all $p$. One can show that $p\fo |\alpha|\in A$ for all $p$.  If $\alpha$ is $F_a^A$, then $p\fo a\in A$ by definition, and so $p\fo F_a\in A$.  Consider an application term $\alpha\beta$. Say $\alpha$ is of type $A\to B$ and $\beta$ is of type $A$. We claim that $p\fo \alpha\beta\in B$. Let $q\supseteq p$. By induction, we can find $r\supseteq q$ such that $r\fo |\alpha|\in A\to B$ and $r\fo |\beta|\in A$. Here we use monotonicity. By definition, there is $s\supseteq r$ such that $s\fo |\alpha|_r^s|\beta|_r\in B$. Since $|\alpha\beta|_s\simeq |\alpha|_r^s|\beta|_r$, the claim follows. 
\end{proof}
\begin{proof}[Proof of (T4)]
	If $\gamma$ is $x^A$ then it is trivial. If $\gamma$ is a constant see (T3).  
	Let $\gamma$ be an application term, say $(\gamma_1\gamma_2)(x^A)$ of type $C$ where $\gamma_i$ is of type $A_1=B\to C$ and $\gamma_2$ is of type $A_2=B$.  By induction, $p\fo \gamma_i(\alpha)=_{A_i}\gamma_i(\beta)$. Let us show that $p\fo \gamma_1(\alpha)\gamma_2(\alpha)=_C\gamma_1(\beta)\gamma_2(\beta)$. Let $q\supseteq p$. By induction we can find $r\supseteq q$ such that $r\fo |\gamma_i(\alpha)|=_{A_i}|\gamma_i(\beta)|$. Note that here we use monotonicity. By definition, we can find $s\supseteq r$ such that $s\fo |\gamma_1(\alpha)|_r^s|\gamma_2(\alpha)|_r=_B|\gamma_1(\beta)|_r^s|\gamma_2(\beta)|_r$. Since $|\gamma(\alpha)|_s\simeq |\gamma_1(\alpha)|_r^s|\gamma_2(\alpha)|_r$ and similarly  $|\gamma(\beta)|_s\simeq |\gamma_1(\beta)|_r^s|\gamma_2(\beta)|_r$, the claim follows. 
\end{proof}
\begin{proof}[Proof of (T5)]
	Suppose that $p\fo \gamma(\alpha)=_B\delta(\alpha)$ and $p\fo \alpha=_A\beta$. By (T4), it follows that $p\fo \gamma(\alpha)=_B\gamma(\beta)$ and $p\fo \delta(\alpha)=_B\delta(\beta)$. By symmetry and transitivity, we get $p\fo \gamma(\beta)=_B\delta(\beta)$.  
\end{proof}

\tb{Remark}. One can define, along the lines of Goodman, $p\fo \alpha=_A\beta$ iff $|\alpha|_p$ and $|\beta|_p$ are both defined and $p\fo |\alpha|_p=_A|\beta|_p$. Both definitions are fine, but we get slightly different properties.  For instance, by our definition, if  $p\fo \alpha=_A\beta$ then $p\fo \gamma(\alpha)=_B\gamma(\beta)$. This is property (T4). By this alternative definition, we only get that for some $q\supseteq p$, $q\fo \gamma(\alpha)=_B\gamma(\beta)$. This is however a minor issue.  \\

(3) For every $p\in T$, for all numbers $a,b$ and for every sentence $\varphi$ of $\mathcal L_p$ we define  $p\fo (a,b)\colon \vp$. The  meaning is that $p$ forces  $a$ and $b$ to be  equal realizers of $\vp$. Let $p\fo a\colon \vp$ be  a shorthand for $p\fo (a,a)\colon \vp$.  

\begin{itemize}
	\item  $p\fo (a,b)\colon \alpha=_A\beta$ iff $p\fo \alpha=_A\beta$. 
	\item $p\fo (a,b)\colon \vp\land \psi$ iff $p\fo ((a)_0,(b)_0)\colon \vp$ and $p\fo ((a)_1,(b)_1)\colon \psi$
	\item $p\fo (a,b)\colon \vp\lor\psi$ iff either  $(a)_0=(b)_0=0$ and $p\fo ((a)_1,(b)_1)\colon \vp$ or else $(a)_0=(b)_0=1$ and $p\fo ((a)_1,(b)_1) \colon\psi$
	\item $p\fo (a,b)\colon \vp\imp\psi$ iff for all $q\supseteq p$, if $q\fo (c,d)\colon \vp$ then there is $r\supseteq q$ such that $r\fo (ac,bd) \colon \psi$. 
	\item $p\fo (a,b)\colon \exists x^A\vp$ iff $p\fo (a)_0=_A(b)_0$ and $p\fo ((a)_1,(b)_1)\colon \vp(F_{(a)_0}^A)$  
	\item $p\fo (a,b)\colon \forall x^A\vp$ iff for all $q\supseteq p$, if $q\fo n=_Am$ then
	there is $r\supseteq q$ such that  $r\fo (an,bm) \colon \vp(F_n^A)$
\end{itemize}
Compare this definition with Goodman realizability. The atomic case is based on a different interpretation of equality: $\sf HRO$ in Goodman and $\sf HEO$ here. Indeed, one can recover Goodman realizability by reading $p\fo \alpha=_A\beta$ as  $|\alpha|_p=|\beta|_p$ and  $p\fo a=_Ab$ as $a=b$ and $p\fo a\in A$.  On the other hand, one cannot recover our interpretation from Goodman realizability by simply replacing the atomic case. The treatment of equality between realizers is indeed the main feature of our interpretation.\smallskip

Properties.  (R1) Monotonicity. If $p\fo (a,b)\colon\vp$ and $q\supseteq p$ then $q\fo (a,b)\colon \vp$. (R2) If $p\fo (a,b)\colon \vp(\alpha)$ and $p\fo \alpha=_A\beta$, then  $p\fo (a,b)\colon \vp(\beta)$. By monotonicity. Use (T5) in the atomic case. In particular, if $p\fo (a,b)\colon \vp(F_n)$ and $p\fo n=_Am$ then $p\fo (a,b)\colon \vp(F_m)$.   (R3) Reflexivity. If  $p\fo (a,b)\colon \vp$ then  $p\fo a,b\colon \vp$.  Symmetry. If $p\fo (a,b)\colon\vp$ then $p\fo (b,a)\colon \vp$.  Transitivity. If $p\fo (a,b)\colon \vp$ and $p\fo (b,c)\colon \vp$ then $p\fo (a,c)\colon \vp$. By induction on $\vp$.  Use (R2) in $\exists$ and $\forall$.\\

By $a\colon \vp$ we mean that  $p\fo a\colon\vp$ for all $p$. In the proof of the soundness we  use the  following property to verify that $a\colon \vp$. \smallskip

(S) Let $\vp$ be $\forall x_1\ldots \forall x_k(\vp_1\imp\ldots \imp \vp_{l+1})$. Say $x_i$ has type $A_i$. 
\begin{itemize}
	\item Suppose $(a,b)\colon \vp$. Let $p\fo n_i=_{A_i}m_i$ for $i=1,\ldots,k$ and $p\fo (c_i,d_i)\colon \vp_i(F_{n_1},\ldots , F_{n_k})$ for $i=1,\ldots,l$.  Then there is $q\supseteq p$ such that 
	\[   q\fo (an_1\ldots n_kc_1\ldots c_l,bm_1\ldots m_kd_1\ldots d_l)\colon \vp_{l+1}(F_{n_1},\ldots F_{n_k}). \]
	\item Suppose $a^pn_1\ldots n_kb_1\ldots b_{l-1}$ is defined for all $n_i,b_i,p$. Moreover, suppose that  if $p\fo n_i=_{A_i}m_i$ for $i=1,\ldots,k$ and $p\fo (b_i,c_i)\colon \vp_i(F_{n_1},\ldots , F_{n_k})$ for $i=1,\ldots,l$,  then there is $q\supseteq p$ such that 
	\[  q\fo (an_1\ldots n_kb_1\ldots b_l,am_1\ldots m_kc_1\ldots c_l)\colon \vp_{l+1}(F_{n_1},\ldots F_{n_k}). \]
	Then $a\colon \vp$.  
\end{itemize} 

Note that  the following instances of (S) hold true (\footnote{Actually, (S1) does not follow from (S) but the proof is similar.}). \smallskip

(S1) Let $\vp$ be $\forall x_1\ldots \forall x_k\vp$. Say $x_i$ has type $A_i$. Then $(a,b)\colon \vp$ iff for all $n_i,m_i$ and for all $p$ such that $p\fo n_i=_{A_i}m_i$ for $i=1,\ldots, k$, there is $q\supseteq p$ such that $q\fo (an_1\ldots n_k,bm_1\ldots m_k)\colon \vp(F_{n_1},\ldots,F_{n_k})$.\smallskip

(S2) Let $\vp$ be $\vp_1\imp\ldots\imp \vp_l\imp \vp_{l+1}$. 
\begin{itemize}
	\item Suppose $(a,b)\colon \vp$ and $p\fo (c_i,d_i)\colon \vp_i$ for all $i=1,\ldots,l$. Then there is $q\supseteq p$ such that $q\fo (ac_1\ldots c_l,ad_1\ldots d_l)\colon \vp_{l+1}$. 
	\item Suppose that $a^pb_1\ldots b_{l-1}$ is defined for all $b_i,p$. Moreover, suppose that whenever $p\fo (b_i,c_i)\colon \vp_i$ for all $i=1,\ldots,l$,   there is $q\supseteq p$ such that $q\fo (ab_1\ldots b_l,ac_1\ldots c_l)\colon \vp_{l+1}$. Then $a\colon \vp$.
\end{itemize}

\begin{theorem}[Soundness]\label{sound}
	Suppose $\vp$ is a formula of $\mcal L$ such that \[\ehaw+\ac+\dc\vdash \varphi.\] If the free variables of $\vp$ are among the distinct variables $x_1,\ldots, x_k$, then there exists an index  $a$ such that for every set $T$ of partial functions, $p\fo a\colon \forall x_1\ldots\forall x_k\vp$ for every $p\in T$. In particular, if $\vp$ is a sentence, then there exists an index $a$ such that for every set $T$ of partial functions, $p\fo a\colon \vp$ for every $p\in T$.	
\end{theorem}

\begin{proof}
	By induction on the proof of $\vp$. The indices are the same as in Goodman \cite[Theorem 1]{goodman}.
	
	Say $x_i$ is of type $A_i$. For the sake of brevity, we write $A$ for $A_1,\ldots, A_k$ and  $x^A$ for $x_1,\ldots,x_k$. For the remainder of the proof we assume that the free variables of $\vp$ are among $x^A$. We thus write $a\colon \forall x^A\vp$ or simply $a\colon \vp$ for $p\fo a\colon \forall x_1\ldots\forall x_k\vp$ for every $p\in T$.  To ease notation, we write, e.g., $n$ for $n_1\ldots n_k$ so that $a^pn$ stands for $a^pn_1\ldots n_k$ and $p\fo n=_Am$ stands for $p\fo n_1=_{A_1}m_1, \ldots, p\fo n_k=_{A_k}m_k$. Similarly, we write $\vp(F_n)$ for $\vp(F_{n_1},\ldots,F_{n_k})$. \\
	
	We repeatedly use (S) and the fact that Kleene application  and the forcing relation are monotone. \\	
	
	(1) $a\colon \vp\imp\vp$, where $a^pnb=b$.
	
	By (S), it is enough to show that if $p\fo n=_Am$ and $p\fo (b,c)\colon \vp(F_n)$, then there is $q\supseteq p$ such that $q\fo (anb,amc)\colon \vp(F_n)$. Easy. In fact, $p\fo (anb,amc)\colon\vp(F_n)$.  \\
	
	(2) Modus ponens: if $c\colon \vp$ and $b\colon \vp\imp\psi$ then $a\colon \psi$.
	
	As in Goodman \cite{goodman} we have to deal with variables, say $x^B$, that might occur free in $\vp$ but are not among $x^A$. Then the free variables of $\vp$ and $\vp\imp\psi$ are among $x^A,x^B$. By induction, let  $c\colon \forall x^A\forall x^B\vp$ and $b\colon\forall x^A\forall x^B(\vp\imp\psi)$. Choose $a$  such that $a^pn\simeq b^pn0^B(c^pn0^B)$, where $p\fo 0^B\in B$ for all $p$. We claim that $a\colon\forall x^A\psi$. 
	
	Suppose $p\fo n=_Am$. By (S1), it suffices to show that there is $q\supseteq p$ such that $q\fo (an,am)\colon \psi(F_n)$. Since $p\fo 0^B\in B$, we can find $q\supseteq p$ such that $q\fo (cn0^B,cm0^B)\colon \vp(F_n,F_{0^B})$. Therefore there is $r\supseteq q$ such that $r\fo (bn0^B(c^rn0^B),bm0^B(c^rm0^B))\colon \psi(F_n)$. Then $r$ is as required. \\
	
	(3) Syllogism: if $c\colon \vp\imp\psi$ and $b\colon \psi\imp\chi$ then $a\colon \vp\imp\chi$.
	
	As in (2) there might be variables, say $x^B$, that are free in $\psi$ but do  not occur in $x^A$. Let $c\colon \forall x^A\forall x^B(\vp\imp\psi)$ and $b\colon \forall x^A\forall x^B(\psi\imp\chi)$. Let $0^B$ be such that $p\fo 0^B\in B$ for all $p$. Choose $a$ such that $a^pnd\simeq b^pn0^B(c^pn0^Bd)$. We claim that $a\colon \forall x^A(\vp\imp\psi)$.
	
	Note that $a^pn$ is defined for all $n,p$. By (S), it is enough to show that if $p\fo n=_Am$ and $p\fo (d,e)\colon \vp(F_n)$,  then there is $q\supseteq p$ such that $q\fo (and,ame)\colon \chi(F_n)$. By induction, since $p\fo 0^B\in B$, that is, $p\fo 0^B=_B0^B$, we can find $q\supseteq p$ such that $q\fo (cn0^Bd,cm0^Be)\colon \psi(F_n)$. Therefore we can find $r\supseteq q$ such that $r\fo (bn0^B(c^rn0^Bd), bm0^B(c^rm0^Be))\colon \chi(F_n)$. Then $r$ is as required. \\

	(4) $a\colon \psi_0\land\psi_1\imp\psi_i$, where $a^pnb=(b)_i$.
	
	By (S), it is enough to show that if $p\fo n=_Am$ and $p\fo (b,c)\colon (\psi_0\land\psi_1)(F_n)$, then there is $q\supseteq p$ such that $q\fo (anb,amc)\colon \psi_i(F_n)$. Easy. We have that $p\fo ((b)_i,(c)_i)\colon \psi_i(F_n)$, and so $p$ is as required. \\
	
	(5) $a\colon \psi_i\imp\psi_0\lor\psi_1$, where $a^pnb=\langle i,b\rangle$.
	
	Suppose $p\fo n=_Am$ and $p\fo (b,c)\colon \psi_i(F_n)$. Then $p\fo (\pair ib,\pair ic)\colon (\vp\lor\psi)(F_n)$, that is, $p\fo (anb,amc)\colon (\vp\lor\psi)(F_n)$. By (S), $a$ is as desired.\\ 
	
	(6) If $b\colon \chi\imp\vp$ and $c\colon \chi\imp\psi$ then $a\colon \chi\imp\vp\land\psi$, where $a^pnd\simeq \langle b^pnd,c^pnd\rangle$.
	
	First, note that the free variables of the premises are among $x^A$. Note also that $a^pn$ is defined for all $n,p$. By (S), it is enough to show that if  $p\fo n=_Am$ and $p\fo (d,e)\colon \chi(F_n)$,  then there is $q\supseteq p$ such that $q\fo (and,ame)\colon (\vp\land\psi)(F_n)$. Easy. By induction, we can find $q\supseteq p$ such that $q\fo (bnd,bme)\colon \vp(F_n)$ and $q\fo (cnd,cme)\colon \psi(F_n)$. Then $q\fo (\pair {bnd}{cnd},\pair{bme}{cme})\colon (\vp\land\psi)(F_n)$. So $q$ is as desired. \\
	
	(7) If $b\colon \vp\imp\chi$ and $c\colon\psi\imp\chi$ then $a\colon \vp\lor\psi\imp\chi$, where  \[ a^pnd\simeq \begin{cases}  b^pn(d)_1 &  \text{ if } (d)_0=0\\
	c^pn(d)_1 &   \text{ otherwise}
	\end{cases} \]
	
	Note that $a^pn$ is defined for all $n,p$. By (S), it is enough to show that if  $p\fo n=_Am$ and $p\fo (d,e)\colon (\vp\lor\psi)(F_n)$, then  there is $q\supseteq p$ such that $q\fo (and,ame)\colon \chi(F_n)$. Suppose $(d)_0=0=(e)_0$ and $p\fo ((d)_1,(e)_1)\colon \vp(F_n)$. Then we can find $q\supseteq p$ such that $q\fo (bn(d)_1,bm(e)_1)\colon \chi(F_n)$. But $a^qnd=b^qn(d)_1$ and $a^qme=b^qm(e)_1$. Then $q$ is as desired. The case $(d)_0=1=(e)_0$ is similar. \\
	
	(8) If $b\colon \vp\land\psi\imp\chi$ then  $a\colon \vp\imp(\psi\imp\chi)$, where $a^pnce\simeq b^pn\langle c,e\rangle$.
	If $b\colon \vp\imp(\psi\imp\chi)$ then $a\colon  \vp\land\psi\imp\chi$, where $a^pnc\simeq b^pn(c)_0(c)_1$.
	
	For the first rule, note that $a^pnc$ is defined for all $n,c,p$. By (S), it is enough to show that if  $p\fo n=_Am$, $p\fo (c,d)\colon \vp(F_n)$ and $p\fo (e,f)\colon \psi(F_n)$, then  there is $q\supseteq p$ such that $q\fo (ance,amdf)\colon \chi(F_n)$. We have $p\fo (\pair ce , \pair df)\colon (\vp\land\psi)(F_n)$. By induction,  there is $q\supseteq p$ such that $q\fo (bn\pair ce,bm\pair df)\colon \chi(F_n)$. Then $q$ is as required. 
	
	The second rule is similar.  \\
	
	(9) $a\colon 0=S0\imp\vp$, where  $a^pnb=0$. 
	
	Easy. Note that no $p$ forces $0=S0$.\\
	
	(10) If $b\colon\vp\imp\psi$ then $a\colon\vp\imp\forall x^B\psi$, for $x$ not free in $\vp$.
	
	Suppose that $x^B$ is not among $x^A$. Then the free variables of $\vp\imp\psi$ are among $x^A,x^B$.  Let $b\colon \forall x^A\forall x^B(\vp\imp\psi)$. Choose $a$ such that $a^pnce\simeq b^pnec$. We claim that $a\colon \forall x^A(\vp\imp\forall x^B\psi)$.
	
	Note that $a^pnc$ is defined for all $n,c,p$. By (S), it is enough to show that if $p\fo n=_Am$ and  $p\fo (c,d)\colon \vp(F_n)$ then there is $q\supseteq p$ such that $q\fo (anc,amd)\colon \forall x^B\psi(F_n)$. We claim that $p\fo (anc,amd)\colon \forall x^B\psi(F_n)$. Let $q\supseteq p$ such that $q\fo e=_Bf$. We aim to find $r\supseteq q$ such that $r\fo (ance,amdf)\colon \psi(F_n,F_e)$. By induction, we can find $r\supseteq q$ such that $r\fo (bnec,bmdf)\colon \psi(F_n,F_e)$. Then $r$ is as required.  
	
	If $x^B$ is among $x_1,\ldots,x_k$, say $x^B$ is $x_i$, then the free variables of $\vp\imp\psi$ are also among $x_1,\ldots, x_k$. Choose $a$ such that $a^pnce\simeq b^pn_1\ldots n_{i-1}en_{i+1}\ldots n_kc$. Then $a$ is as desired. The proof is similar.  \\
	
	(11) If $b\colon\vp\imp\psi$ then $a\colon\exists x^B\vp\imp\psi$, for $x$ not free in $\psi$.
	
	As in (10), suppose first that $x^B$ is not among $x_1,\ldots, x_k$. Then the free variables of $\vp\imp\psi$ are among $x_1,\ldots,x_k,x$.  Let $b\colon \forall x_1\ldots \forall x_k\forall x\ (\vp\imp\psi)$. Choose $a$ such that $a^pnc\simeq b^pn(c)_0(c)_1$. 
	
	By (S), it is enough to show that if $p\fo n=_Am$ and $p\fo (c,d)\colon \exists x^B\vp(F_{n},x^B)$, then   there is $q\supseteq p$ such that $q\fo (anc,amd)\colon \psi(F_n)$. As $p\fo (c)_0=_B(d)_0$, there is $q\supseteq p$ such that $q\fo (bn(c)_0,bm(d)_0)\colon \vp(F_n,F_{(c)_0})\imp\psi(F_n)$. As $q\fo ((c)_1,(d)_1)\colon \vp(F_n,F_{(c)_0})$, we can find $r\supseteq q$ such that $r\fo (bn(c)_0(c)_1,bm(d)_0(d)_1)\colon \psi(F_n)$, that is, $r\fo (anc,amd)\colon \psi(F_n)$, as desired.  
	
	If $x^B$ is among $x_1,\ldots,x_k$, say $x^B$ is $x_i$, then the free variables of $\vp\imp\psi$ are also among $x_1,\ldots, x_k$. Choose $a$ such that $a^pnc\simeq b^pn_1\ldots n_{i-1}(c)_0n_{i+1}\ldots n_k(c)_1$. Then $a$ is as desired. The proof is similar. \\

	(12) $a\colon \forall x^B\vp\imp\vp(\alpha)$ and $a\colon \vp(\alpha)\imp\exists x^B\vp$, for any term $\alpha$ free for $x$ in $\vp$, where  $a^pnb\simeq b^p(d^pn)$ and   $a^pnb\simeq \langle d^pn,b\rangle$ respectively, and $d^pn\simeq |\alpha(F_n)|_p$. The index $d$ exists by (V2). 
	
	For the first axiom, it is sufficient to show by (S) that if $p\fo n=_Am$ and $p\fo (b,c)\colon \forall x^B\vp(F_n,x^B)$, then there is $q\supseteq p$ such that $q\fo (anb,amc)\colon \vp(F_n,\alpha(F_n))$. By (T4), as $p\fo F_n=_AF_m$, we have that $p\fo \alpha(F_n)=_B\alpha(F_m)$. Let $q\supseteq p$ be such that $q\fo |\alpha(F_n)|=_B |\alpha(F_m)|$, that is, $q\fo dn=_Bdm$.  As $q\fo (b,c)\colon \forall x^B\vp(F_n,x^B)$, there is $r\supseteq q$ such that $r\fo (b(dn),c(dm))\colon \vp(F_n,F_{dn})$. As $|F_{d^rn}|_r=d^rn=|\alpha(F_n)|_r$, by (R2) we have that  $r\fo (b(dn),c(dm))\colon \vp(F_n,\alpha(F_n))$, and thus $r\fo (anb,amc)\colon \vp(F_n,\alpha(F_n))$. So $r$ is as required. 
	
	For the second axiom, it is sufficient to show by (S) that if $p\fo n=_Am$ and $p\fo (b,c)\colon \vp(F_n,\alpha(F_n))$, then there is $q\supseteq p$ such that $q\fo (anb,amc)\colon \exists x^B\vp(F_n,x)$.  By (T4), as $p\fo F_n=_AF_m$, we have that $p\fo \alpha(F_n)=_B\alpha(F_m)$. Let $q\supseteq p$ be such that $q\fo |\alpha(F_n)|=_B |\alpha(F_m)|$, that is, $q\fo dn=_Bdm$. As $q\fo \alpha(F_n)=_BF_{dn}$, by (R2) we have that  $q\fo (b,c)\colon \vp(F_n,F_{dn})$ and so $q\fo (anb,amc)\colon \exists x^B\vp(F_n,x)$. \\

	(13) $a\colon\neg(0=Sx)$ and $a\colon Sx=Sy\imp x=y$, for  $x,y$ variables of type $N$, where $a^pnb=0$.
	
	Easy.\\
	
	All the axioms from (14) to  (18) are given by atomic formulas of the form $\alpha(x^A)=_B\beta(x^A)$.\\
	
	(14) Defining equations of primitive recursive functions.
	
	(15) Combinators: $\varPi xy=x$ and $\varSigma xyz=(xz)(yz)$.
	
	(16) Recursors: $Rxy0=x$ and $Rxy(z+1)=y(Rxyz)z$ at all sensible types
	
	(17) $D_0(D xy)=x$ and $D_1(Dxy)=y$, for $x$ of type $A$ and $y$ of type $B$
	
	(18) $x=D(D_0x)(D_1x)$, for  $x$ of type $A\times B$\\
	
	Choose $a$ such that $a^pn=0$ for all $n,p$. It clearly suffices to prove that $(*)$ if $p\fo n\in A$ then $p\fo \alpha(F_n)=_B\beta(F_n)$.
	The verification is routine. Consider for instance the axioms for the recursors.
	
	(16) $Rxy0=_Bx$ and $Rxy(Sz)=_By(Rxyz)z$, for $x$ of type $B$, $y$ of type $B\to N\to B$ and $z$ of type $N$. 
	
	Note that $(*)$ is equivalent to saying that for every $p$, if $p\fo n\in A$, then there is $q\supseteq p$ such that $q\fo |\alpha(F_n)|=_B|\beta(F_n)|$. 
	
	We first claim that if $p\fo a\in B$ and $p\fo b\in B\to N\to B$, then for all $i$ there exists $q\supseteq p$ such that $q\fo |RF_aF_bF_i|\in B$. Let $i=0$. Then $|RF_aF_bF_0|_p\simeq |R|^pab0$. By definition, $|R|^pab0=a$, and so $p$ is as desired. Case $i+1$. Let $q\supseteq p$ such that $q\fo |RF_aF_bF_i|\in B$. By assumption, we can find $r\supseteq q$ such that $r\fo b(|RF_aF_bF_i|_q)i\in B$. On the other hand, $|RF_aF_bF_{i+1}|_q\simeq b^q(|R|^qabi)i$. It follows that $r$ is as desired. 
	
	For the first axiom, suppose that $x^A$ is $x,y,\ldots$ Let $p\fo n\in A$. In particular, $p\fo n_1\in B$. We claim that $p\fo RF_{n_1}F_{n_2}0=_BF_{n_1}$. By definition, $|R|^pn_1n_20=n_1$. On the other hand, $|RF_{n_1}F_{n_2}0|_p\simeq |R|^pn_1n_20$.  By (T3), it follows that  $p$ is as desired. 
	
	For the second axiom, suppose that $x^A$ is $x,y,z,\ldots$ Let $p\fo n=_Am$. In particular, $p\fo n_1\in A$ and $p\fo n_2\in B\to N\to B$. By the claim, there is $q\supseteq p$ such that $q\fo |RF_{n_1}F_{n_2}F_{n_3}|\in B$. By the assumption on $b$, we can find $r\supseteq q$ such that $r\fo n_2(|RF_{n_1}F_{n_2}F_{n_3}|_qn_3\in B$. It is not difficult to check by using the definition of $|R|$ that $|RF_{n_1}F_{n_2}(SF_{n_3})|_r=|F_{n_2}(RF_{n_1}F_{n_2}F_{n_3})F_{n_3}|_r$. It follows by (T3) that $r$ is as desired. \\

	(19) Induction: If $b\colon\vp(0)$ and $c\colon \vp(x)\imp\vp(Sx)$ then $a\colon \vp(x)$, for $x$ of type $N$, where  $a^pn0\simeq b^pn$ and $a^pn(i+1)\simeq  c^pni(a^pni)$. 
	
	We may assume that $x^A$ is $x^B,x^N$.  Let $b\colon \forall x^B\vp(0)$ and  $c\colon \forall x^B\forall x^N(\vp(x)\imp \vp(Sx))$. We claim that $a\colon \forall x^B\forall x^N\vp(x)$. 
	
	By (S1), it is enough to show that for all $n,m,i,p$, if  $p\fo n=_Bm$ then there is $q\supseteq p$ such that $q\fo (ani,ami)\colon \vp(F^B_n,F^N_i)$. By induction on $i$. Let $i=0$. Fix $n,m,p$ such that $p\fo n=_Bm$. Then there is $q\supseteq p$ such that $q\fo (bn,bm)\colon \vp(F_n,0)$. By (R2),  $q\fo (bn,bm)\colon \vp(F_n,F_0)$. Thus $q\fo (an0,am0)\colon \vp(F_n,F_0)$. Suppose this is true for $i$ and let $p\fo n=_Bm$. By induction, there is $q\supseteq p$ such that $q\fo (ani,ami)\colon \vp(F_n,F_i)$. By definition, we can find $r\supseteq q$ such that $r\fo (cni(ani),cmi(ami))\colon \vp(F_n,SF_i)$. Since $|SF_i|_r=i+1$, by (R2) we have that $r\fo (cni(ani),cmi(ami))\colon \vp(F_n,F_{i+1})$, and therefore $r\fo (an(i+1),am(i+1))\colon \vp(F_n,F_{i+1})$, as desired. \\

	(20) $a\colon x=x$ for $x$ of type $A$, where $a^pn=0$.
	
	Easy. \\
	
	(21) $a\colon x=y\lor \neg (x=y)$ for $x$ and $y$ of type $N$.
	We may assume  that $x^A$ is $x,y,\ldots$ Choose $a$ such that $a^pn_1n_2\ldots n_k= \langle 0,0\rangle$ if $n_1=n_2$, $\langle 1,0\rangle$ otherwise. 
	
	Easy. \\
	
	(22) Leibniz: $a\colon x^B= y^B\land \varphi(x)\imp \varphi(y)$, for $x$ and $y$ of type $B$, where $a^pnb= (b)_1$. 
	
	We may assume that $x^A$ is $x,y\ldots$  Let $p\fo n=_{A}m$  and $p\fo (b,c)\colon F_{n_1}=_BF_{n_2}\land \vp(F_{n})$. By (S), it suffices to show that there is $q\supseteq p$ such that $q\fo (anb,amc)\colon \vp(F_{n_2},F_{n_2},\ldots,F_{n_k})$, that is, $q\fo ((b)_1,(c)_1)\colon \vp(F_{n_2},F_{n_2},\ldots,F_{n_k})$. By definition, we have $p\fo ((b)_1,(c)_1)\colon \vp(F_{n})$. On the other hand, $p\fo ((f)_0,(g)_0)\colon F_{n_1}=_BF_{n_2}$, that is, $p\fo F_{n_1}=_BF_{n_2}$. By (R2)  the claim follows. \\

	(23) Extensionality:  $a\colon \forall z^B(xz=_Cyz)\imp x=_{B\to C}y$, where  $a^pnb=0$. 
	
	In this case $A$ is $B\to C$. We may assume that $x^A$ is $x,y,\ldots$  Suppose $p\fo n=_Am$ and $p\fo (b,c)\colon \forall z^B(F_{n_1}z=_C F_{n_2}z)$.  By (S), it suffices to show that $p\fo (anb,amc)\colon F_{n_1}=_{B\to C} F_{n_2}$, that is, $p\fo F_{n_1}=_{B\to C}F_{n_2}$. As $|F_{n_i}|_p=n_i$, it is enough to show that $p\fo n_1=_{B\to C}n_2$.  Note that $p\fo n_1,n_2\in B\to C$. By (E3), we just need to show that if $q\supseteq p$ and $q\fo d\in B$ then there is $r\supseteq q$ such that $r\fo n_1d=_Cn_2d$. Let $q\fo d\in B$. As $q\fo (b,c)\colon \forall z^B(F_{n_1}z=F_{n_2}z)$, we can find $r\supseteq q$ such that 
	$r\fo (bd,cd) \colon F_{n_1}F_d=_CF_{n_2}F_d$, that is, $r\fo  F_{n_1}F_d=_CF_{n_2}F_d$.  By definition, we can find $s\supseteq r$ such that $s\fo n_1d=_Cn_2d$.\\
	
	(24) Axiom of choice: $a\colon \forall x^B\exists y^{C}\varphi(x,y)\imp \exists z^{B\to C}\forall x^B\varphi(x,zx)$, where
	$a^pnb=\pair{{a_0}^pnb}{a_1^pnb}$, and $a_i^pnbd\simeq (b^pd)_i$.

	Let $p\fo n=_Am$ and $p\fo (b,c)\colon \forall x^B\exists y^C\vp(F_n,x,y)$. We claim that \[p\fo (anb,amc)\colon \exists z^{B\to C}\forall x^B\vp(F_n,x,zx),\]  that is,
	\begin{enumerate}
		\item[(i)] $p\fo a_0nb=_{B\to C}a_0mc$
		\item[(ii)]  $p\fo (a_1nb,a_1mc)\colon \forall x^B\vp(F_n,x,F_{a_0nb}x)$
	\end{enumerate}
	For (i), let $q\supseteq p$ and suppose that $q\fo d=_Be$. As $q\fo (b,c)\colon \forall x^B\exists y^C\vp(F_n,x,y)$, we can find $r\supseteq q$ such that  $r\fo (bd,ce)\colon \exists y^C\vp (F_n,F_d,y)$.
	In particular,  $r\fo (bd)_0=_C(ce)_0$, and so  $r\fo a_0nbd=_Ca_0mce$, 
	as desired.  This proves (i). 
	
	For (ii), let $q\supseteq p$ and suppose that $q\fo d=_Be$.  Since $q\fo (b,c)\colon \forall x^B\exists y^C\vp(F_n,x,y)$,  we can find $r\supseteq q$ such that $r\fo  ((bd)_1,(ce)_1)\colon\vp(F_n,F_d,F_{(bd)_0})$
	We claim that
	\[   r\fo (a_1nbd,a_1mce)\colon\vp(F_n,F_d,F_{a_0nb}F_d). \]
	
	Now, $|F_{a_0^rnb}F_d|_r\simeq |(a_0^rnb)^rd|\simeq a_0^rnbd\simeq (b^rd)_0$ and $|F_{(b^rd)_0}|_r= (b^rd)_0$.  By  (R2) it follows that 
	\[ r\fo   ((bd)_1,(ce)_1)\colon\vp(F_n,F_d,F_{a_0nb}F_d).   \] 
	The  claim follows. By (S), $a$ is as desired. \\

	(25) Axiom of relativized dependent choice: \[ a\colon \forall x^B[\vp(x)\imp \exists y^B(\vp(y)\land \psi(x,y))]\imp \forall x^B[\vp (x)\imp \exists z^{N\to B}(z0=x\land \forall v^N\psi(zv,z(Sv)))].\]  
	
	For ease of presentation, we assume that $x^A$ is empty.  Let $a^pbnd\simeq \langle f^pbnd, 0,g^pbnd \rangle$, where
	$f^pbndi\simeq (h^pbndi)_0$, $g^pbndi\simeq (h^pbnd(i+1))_2$, and $h$ is such that $h^pbnd0=\langle n,d,0\rangle$ and \[h^pbnd(i+1)\simeq b^p(h^pbndi)_0(h^pbndi)_1.\] 
	We claim that $a$ is as desired. 
	Note tha $a^pbnd$ is defined for all $b,n,d,p$. We use (S). Suppose $p\fo (b,c)\colon \forall x^B[\vp(x)\imp \exists y^B(\vp(y)\land\psi(x,y))]$, $p\fo n=_Bm$ and $p\fo (d,e)\colon \vp(F_n)$. We claim that  $p\fo (abnd,acme)\colon \exists z^{N\to B}(z0=F_n\land \forall v^N\psi (zv,z(Sv))$.  Note that $(a^pbnd)_1=0=(acme)_1$. It is thus sufficient  to show that:
	\begin{enumerate}
		\item[(i)] $p\fo  F_{fbnd}0=_B F_n$.
		\item[(ii)] $p\fo fbnd=_{N\to B}fcme$
		\item[(iii)]  $p\fo (gbnd,gcme)\colon  \forall v^N\psi(F_{fbnd}v,F_{fbnd}(Sv))$.
	\end{enumerate}
	Since $p\fo n\in B$, $|F_{f^pnd}0|_p\simeq f^pnd0$,  and $f^pnd0=n$, we have by (T4) that $p\fo F_{fbnd}0=_B F_n$, and so (i) holds. It remains to show (ii) and (iii).  To this end, we claim that for all $q\supseteq p$ and for all $i$ there is $r\supseteq q$ such that 
	$r\fo fbndi=_B fcmei$ and $r\fo ((hbndi)_{10},(hcmei)_{10})\colon \vp(F_{fbndi})$. By induction on $i$. Let $i=0$ and $q\supseteq p$. It is easy to see that $q$ is as desired. In fact, by monotonicity, $q\fo n=_Bm$ and $q\fo (d,e)\colon \vp(F_n)$. On the other hand, we have that  $f^qbnd0=n$, $f^qcme0=m$, $(h^qbnd0)_{10}=d$ and $(h^qcme0)_{10}=e$.  This proves $i=0$. Suppose this is true for $i$. Let us prove that claim for $i+1$. Let $q\supseteq p$. By induction, let $r\supseteq q$ satisfy the claim for $i$. As  $r\fo (b,c)\colon \forall x^B[\vp(x)\imp \exists y^B(\vp(y)\land\psi(x,y))]$, it follows from the definition of $h$ that there is $s\supseteq r$ such that $s\fo (hbnd(i+1),hcme(i+1))\colon \exists y^B(\vp(y)\land \psi(F_{fbndi}, y))$, and so $s$ is as required. This proves the claim.
	
	Clearly, (ii) follows directly from the claim. For (iii), let $q\supseteq p$ and $i$ be given. We aim to show that there is $r\supseteq q$ such that $r\fo (gbndi,gcmei)\colon \psi(F_{fbnd}F_i,F_{fbnd}(SF_i))$.  By the claim, we can find $r\supseteq q$ such that $r\fo fbndi=_B fcmei$ and $r\fo ((hbndi)_{10},(hcmei)_{10})\colon \vp(F_{fbndi})$. As $r\fo (b,c)\colon \forall x^B[\vp(x)\imp \exists y^B(\vp(y)\land\psi(x,y))]$, it follows from the definition of $h$ that there is $s\supseteq r$ such that $s\fo (hbnd(i+1),hcme(i+1))\colon \exists y^B(\vp(y)\land \psi(F_{fbndi}, y))$, and so $s\fo (gbndi,gcmei)\colon \psi(F_{fbndi},F_{fbnd(i+1)})$. As usual, by (R2), we have that $s\fo (gbndi,gcmei)\colon \psi(F_{fbnd}F_i,F_{fbnd}(SF_i))$, and hence $s$ is as required. This completes the proof of (iii). 
\end{proof}

\begin{theorem}[Self-realizability of first-order formulas]\label{self}
	Let $\vp$ be a first-order formula of $\mcal L$ with free variables among the distinct variables  $x_1,\ldots,x_k$.  Then there exist a nonempty set $T$ of finite partial functions and an index  $a$ such that:
	\begin{itemize}
		\item  If $\vp(\bar n_1,\ldots,\bar n_k)$ is true, then for all $p\in T$ there exists  $q\supseteq p$  such that 
		\[ q\fo an_1\ldots n_k\colon \vp(\bar n_1,\ldots,\bar n_k)\]
		\item If  $p\fo (b,c)\colon  \vp(\bar n_1,\ldots,\bar n_k)$, then $\vp(\bar n_1,\ldots,\bar n_k)$ is true.
	\end{itemize}
	In particular, if $\vp$ is a first-order sentence, then there exists a nonempty set $T$ of finite partial functions such that: 
	\begin{itemize}
		\item  if $\vp$ is true, then for every $p\in T$ there exist  $q\supseteq p$ and an index $a$  such that   $q\fo a\colon \vp$
		\item if $p\fo (b,c)\colon\vp$, then  $\vp$ is true
	\end{itemize}   
\end{theorem}

\begin{proof}	
	Let $\vp_1,\ldots,\vp_{l}$ be an enumeration of all subformulas of $\vp$. After renaming of bound variables, if necessary, we can choose distinct variables $\underline{x}^j=x_1^j,\ldots,x^j_{k_j}$ listing all of the free variables of $\vp_j$ so that if $\vp_j$ is $\psi\circ \chi$ then $\underline{x}^j$ are the variables for both $\psi$ and $\chi$ and if $\vp_j$ is $Q x\psi$ then $\underline{x}^j,x$ are the variables for $\psi$. We can safely assume that $x_1,\ldots,x_k$ are the variables for $\vp$

	Let $T$ be the set of all finite functions $p$ such that for all $j\in\{1,\ldots, l\}$ and for all
	$\langle j,n_1,\ldots, n_{k_j}\rangle\in\dom(p)$, if $\vp_j$ is of the form $\psi_0\lor\psi_1$ then either $p(\langle j,n_1,\ldots,n_{k_j}\rangle)=0$ and $\psi_0(\bar n_1,\ldots, \bar n_{k_j})$ is true, or $p(\langle j,n_1,\ldots,n_{k_j}\rangle)=1$ and $\psi_1(\bar n_1,\ldots, \bar n_{k_j})$ is true, and if $\vp_j$ is of the form $\exists x\psi$, then $\psi(\bar n_1,\ldots,\bar n_{k_j},\bar m)$ is true, where $m=p(\langle j,n_1,\ldots n_{k_j}\rangle)$.
	
	By induction on every subformula $\psi$ of $\vp$, we prove that every $\vp_j$ has an index $a$ with respect to variables $\underline{x}^j$.

	For the sake of brevity we write $n$ for $n_1\ldots n_{k_{j}}$. Similarly,  we write $\psi(\bar n)$ for $\psi(\bar n_1,\ldots,\bar n_{k_j})$ and $p(j,n)$ for $p(\langle j,n_1,\ldots, n_{k_j}\rangle)$. \smallskip

	(1) Suppose $\psi$ is $\alpha=_N\beta$. Let $a^pn=0$ for all $p,n$. 
	
	For every first-order closed  term $\alpha$ there exists a number $|\alpha|$ such that $|\alpha|_p=|\alpha|$ for all $p$. Indeed, $|\alpha|$ is the interpretation of $\alpha$ in the standard model. Therefore, for any first-order closed terms $\alpha$ and $\beta$, $\alpha=\beta$ is true iff $|\alpha|=|\beta|$. We thus have that $\psi(\bar n)$ is true iff $|\alpha(\bar n)|=|\beta(\bar n)|$. On the other hand,  $p\fo (b,c)\colon \psi(\bar n)$ iff $p\fo \psi(\bar n)$ iff $|\alpha(\bar n)|=|\beta(\bar n)|$. This proves (1). \\

	(2) Suppose $\psi$ is $\psi_0\land\psi_1$. Let  $a^pn\simeq \langle a_0^pn,a_1^pn\rangle$, where $a_i$ are indices for $\psi_i$. 
	
	Suppose $\psi(\bar n)$ is true and $p\in T$. Then $\psi_i(\bar n)$ is true for every $i<2$. By induction, we can find $q\supseteq p$ such that $q\fo a_i^qn\colon\psi_i(\bar n)$.  Then $q\fo an\colon \psi(\bar n)$. 
	
	If $p\fo (b,c)\colon \psi(\bar n)$, then by induction $\psi_0(\bar n)$ and $\psi_1(\bar n)$ are both true, and so is $\psi(\bar n)$.\\
	
	(3) Suppose $\psi=\vp_j$ is $\psi_0\lor\psi_1$. Let $a_i$  be an index for  $\psi_i$. Choose $a$ such that for all $n$ and for all $p\in T$,
	\[   a^pn\simeq \begin{cases}
	\langle 0,a_0^pn\rangle & \text{ if } p(j,n)=0 \\
	\langle 1, a_1^pn\rangle & \text{ if } p(j,n)=1 
	\end{cases}\] 
	
	Suppose $\psi(\bar n)$ is true and $p\in T$. If $p(j,n)$ is defined, say $p(j,n)=0$, then $\psi_0(\bar n)$ is true and by induction there is $q\supseteq p$ such that  $q\fo a_0n\colon \psi_0(\bar n)$. Then $a^qn=\langle 0,a_0^qn\rangle$ and $q\fo an\colon \psi(\bar{n})$. Suppose $p(j,n)$ is undefined. Since $\psi(\bar n)$ is true, either $\psi_0(\bar n)$ is true or $\psi_1(\bar n)$ is true. Say $\psi_1(\bar n)$ is true. Let $q\supseteq p$ be such that $q(j,n)=1$. Then $q\in T$. By induction, there exists $r\supseteq q$ such that $r\fo a_1n\colon \psi_1(\bar n)$.  As before,   $r\fo an\colon \psi(\bar{n})$.
	
	If $p\fo (b,c)\colon\psi(\bar n)$, then by induction either $\psi_0(\bar n)$ is true or $\psi_1(\bar n)$ is true, and so is $\psi(\bar n)$. \\
	
	(4) Suppose $\psi$ is $\psi_0\imp\psi_1$. Let  $a^pnb\simeq a_1^pn$, where $a_1$ is an index for $\psi_1$.
	
	Suppose $\psi(\bar n)$ is true. Let $p\in T$. Note that $a^pn$ is defined. We claim that $p\fo an\colon \psi(\bar n)$. Let $q\supseteq p$ and suppose that $q\fo (b,c)\colon \psi_0(\bar n)$. By induction, $\psi_0(\bar n)$ is true and so is $\psi_1(\bar n)$. By induction there is $r\supseteq q$ such that $r\fo a_1n\colon \psi_1(\bar n)$. Then $r\fo (anb,anc)\colon \psi_1(\bar n)$.  
	
	Suppose $p\fo (b,c)\colon \psi(\bar n)$. Let us show that $\psi(\bar n)$ is true. Suppose $\psi_0(\bar n)$ is true. We aim to show that $\psi_1(\bar n)$ is true. By induction, there is $q\supseteq p$  such that  $q\fo a_0n\colon\psi_0(\bar n)$. By definition, there is $r\supseteq q$ such that, e.g,  $r\fo ba_0n\colon \psi_1(\bar n)$. By induction, $\psi_1(\bar n)$ is true, as desired.\\

	(5) Suppose $\psi=\vp_j$ is $\exists x\psi_0$. Let $a^pn\simeq\langle p(j,n), a_0^pnp(j,n)\rangle$.
	
	Suppose $\psi(\bar n)$ is true and let $p\in T$. If $p(j,n)$ is defined, then $\psi_0(\bar n,\bar m)$, where $m=p(j,n)$, is true. By induction there is $q\supseteq p$ such that $q\fo a_0nm\colon \psi_0(\bar n,\bar m)$. By (R3), as $|\bar m|_p=|F_m^N|_p$, we have  $q\fo a_0nm\colon \psi_0(\bar n,F_m)$ and hence $q\fo an\colon \psi(\bar n)$.  If $p(j,n)$ is not defined, let $m$ be such that $\psi_0(\bar n,\bar m)$ is true. Such  $m$ exists by hypothesis. Extend $p$ to a function $q$ such that $q(j,n)=m$. Then $q\in T$. By induction there is $r\supseteq q$ such that $r\fo a_0nm\colon \psi_0(\bar n,\bar m)$. As before, $r\fo an\colon \psi(\bar n)$.
	
	If $p\fo (b,c)\colon \psi(\bar n)$, then $p\fo ((b)_1,(c)_1)\colon \psi_0(\bar n, F_m)$, where $m=(b)_0$. As before, $|F_m|_p=m=|\bar m|_p$. Therefore $p\fo ((b)_1,(c)_1)\colon \psi_0(\bar{n},\bar{m})$. By induction $\psi_0(\bar n,\bar m)$ is true and so is $\psi(\bar n)$. \\
	
	(6) Suppose $\psi$ is $\forall x\psi_0$. Let  $a^pnm\simeq a_0^pnm$, where $a_0$ is an index for $\psi_0$. Choose $a$ such that $a^pn$ is always defined.  
	
	Suppose $\psi(\bar n)$ is true. Let $p\in T$. We claim that $p\fo an\colon \psi(\bar n)$. We just need to show that for all $m$ and for all $q\supseteq p$ there is  $r\supseteq q$ such that $r\fo anm\colon \psi(\bar m, F^N_m)$.  Let $m,q$ be given. Since $\psi(\bar n)$ is true, so is $\psi_0(\bar n,\bar m)$. By induction, there is $r\supseteq q$ such that $r\fo a_0nm\colon \psi_0(\bar n,\bar m)$. The claim follows as in (5). 
	
	Finally, suppose $p\fo (b,c)\colon \psi(\bar n)$. Let $m$ be given. By definition, there is $q\supseteq p$ such that $q\fo (bn,cm)\colon \psi(\bar n,F_m)$. Since $|F_m|_q=m=|\bar m|_q$, we thus have that $q\fo (bn,cm)\colon \psi(\bar n,\bar m)$.  By induction, $\psi(\bar n,\bar m)$ is true. This shows that $\psi(\bar n)$ is true. 
\end{proof}

Formalizing the above proofs in $\ha$, we obtain Goodman's theorem for $\ehaw+\ac+\dc$.

\begin{theorem}
	$\ehaw+\ac+\dc$ is conservative over $\ha$. 
\end{theorem}


\section{Realizing choice and extensionality} \label{realizability}
We provide a realizability interpretation  $a\colon \vp$  such that for every sentence $\vp$ of $\mcal L$,  if 
\[    \ehaw+\ac+\dc\vdash \vp, \]
then there exists an index $a$ such that \[ \ha\vdash  a\colon  \vp. \] 

The idea is to combine $\sf HEO$ with a version of Kleene recursive realizability whose realizers are elements of $\sf HEO$.
\[ \ehaw+\ac+\dc  \vdash \vp\ \ \xrightarrow{{\sf HEO}\ +\ \text{Kleene}}\ \  \ha\vdash a\colon \vp\]

One can also define a realizability notion that combines  Kreisel's modified realizability with the $\sf HEO$ interpretation. This way also independence of premise for $\exists$-free formulas \cite[Section 5.1]{kohl99} is realizable.  
Warning: we write  $a\colon \vp$ for ease  of notation, although the meaning is different from Section \ref{extensional}.\\

As before, the discussion is informal. The reader will convince himself that all the definitions and proofs can be carried out in $\ha$. \\ 

Convention.   We write, e.g., $ab=_Acd$  for $ab$ and $cd$ are defined and $ab=_Acd$. Similarly for $ab\in A$ and $ab\colon \vp$.  \\

(1) For numbers $a,b$ and type $A$, we define $a=_Ab$ such that $(A,=_A)$ is a model of $\ehaw$, where $a\in A$ means $a=_Aa$. By recursion on $A$. 

\begin{itemize}
	\item $a=_Nb$ iff $a=b$
	\item $a=_{A\times B} b$ iff $(a)_0=_A(b)_0$ and $(a)_1=_B(b)_1$
	\item $a=_{A\to B} b$ iff $n=_Am$ implies $an=_Bbm$
\end{itemize}

This a version of $\sf HEO$. Properties: (E1) Reflexivity. If $a=_Ab$ then $a,b\in A$. By definition, if $a\in A$ then $a=_Aa$. Symmetry and transitivity. (E2) $a=_{A\to B}b$ iff $a,b\in A\to B$ and for all $n\in A$, $an=_Bbn$.  (E3) For every type $A$ there is a number $0^A\in A$.  Let $0^N=0$, $0^{A\times B}=\pair {0^A}{0^B}$, and $0^{A\to B}$ be an index for the function $\lambda n. 0^B$. \\

(2) Let $\mathcal L^*$ be the language $\mcal L$ of $\haw$ plus a constant $F_a^A$ of type $A$ for every $a\in A$. We simply write $F_a$ when the type is clear from the context. 
Let us define the value $|\alpha|$ of a closed term $\alpha$ of $\mathcal L^*$ by recursion on $\alpha$:
\begin{itemize}
	\item $|0|=0$
	\item $|S|=\ind n.n+1$ is any index for the function $\lambda n.n+1$.
	\item $|g(\alpha_1,\ldots,\alpha_k)|\simeq g(|\alpha_1|,\ldots,|\alpha_k|)$, for $g$ primitive recursive and $\alpha_i$ of type $N$
	\item $|\varPi|=\ind ab.a$ 
	\item $|\varSigma|=\ind abn.(an)(bn)$   
	\item $|R|$ is  such that
	$|R|ab0=a$ and  $|R|(a,b,n+1)\simeq b(|R|abn,n)$
	\item $|D|=\ind ab.\pair ab$ 
	\item $|D_0|=\ind a. (a)_0$ and $|D_1|=\ind a. (a)_1$
	\item $|\alpha\beta|\simeq \kl{|\alpha|}(|\beta|)$
	\item $|F_a^A|= a$
\end{itemize}
$|R|$ exists by the recursion theorem. Properties. (V1) For every closed term $\alpha$ of type $A$, the number $|\alpha|$ is defined. (V2) For any term $\alpha$ with variables among the distinct variables $x_1,\ldots,x_k$, there exists an index $d$ such that for all $n_1,\ldots, n_k$,   $dn_1\ldots n_k\simeq |\alpha(F_{n_1},\ldots, F_{n_k})|$.  \\

For $\alpha$ and $\beta$ closed terms of type $A$, we define $\alpha=_A\beta$ iff $|\alpha|=_A|\beta|$. In particular, $\alpha\in A$ iff $|\alpha|\in A$. \smallskip

Properties. (T1) For every closed term $\alpha$ of type $A$,  $\alpha\in A$ (easy induction on $\alpha$).  (T2) If $\gamma(x^A)$ is a term of type $B$, $\alpha$ and $\beta$ are closed terms of type $A$, and $\alpha=_A\beta$, then $\gamma(\alpha)=_B\gamma(\beta)$. In particular,  if $|\alpha|=|\beta|$ then  $\gamma(\alpha)=_B\gamma(\beta)$. \\

(3) Type of a formula. We define the type of a formula $\varphi$ of $\mathcal L^*$ as follows.
The type of an atomic formula is $N$. Suppose $\vp $ has type $A$ and $\psi$ has type $B$. Then the type of $\varphi\land \psi$ and $\exists x^A\psi(x)$ is $A\times B$. 
The type of $\varphi\lor\psi$ is $N\times A\times B$. The type of $\varphi\imp \psi$ and $\forall x^A\psi(x)$ is $A\to B$.  
Note that if $\varphi(x)$ has type $A$, so does $\varphi(\alpha)$. If $\vp$ has type $A$,  we write $\vp^A$.\\

(4) Realizability. For a number $a$ and a sentence $\varphi$ of $\mathcal L^*$ we define $a\colon \varphi$ by recursion on $\varphi$:
\begin{itemize}
	\item Atomic case: $a \colon \alpha=_A\beta$ iff  $\alpha=_A\beta$ (iff $|\alpha|=_A|\beta|$)
	\item $a\colon \varphi\land \psi$ iff  $(a)_0\colon \varphi$ and  $(a)_1\colon \psi$
	\item $a\colon \varphi^A \lor \psi^B$ iff  $a\in N\times A\times B$ and  either $(a)_0=0$ and $(a)_1\colon \varphi$ or $(a)_0=1$ and  $(a)_2\colon \psi$
	\item $a\colon \varphi^A\imp \psi^B$ iff   $a\in A\to B$ and
	if $b\colon \varphi$ then $ab\colon  \psi$,
	\item $a\colon \exists x^A\varphi(x)$ iff  $(a)_0\in A$ and  $(a)_1\colon \varphi(F^A_{(a)_0})$.
	\item $a\colon \forall x^A\varphi^B(x)$ iff   $a\in A\to B$ and 
	if $n\in A$ then    $an\colon \varphi(F_n^A)$,
\end{itemize}
The clause for $\lor$  bears similarity with Kreisel's modified realizability. One could introduce disjoint types to deal with $\lor$. We prefer this solution. \smallskip

Properties.  
(R1) $a\colon \varphi^A$ implies $a\in A$ (by induction on $\vp$). (R2) If $\alpha=_A\beta$ and $a\colon \varphi(\alpha)$, then $a\colon \varphi(\beta)$.

\begin{proof}[Proof of (R2)]
	By induction on $\varphi$. For the atomic case we have to show that if $\alpha=_A\beta$ and $\gamma(\alpha)=_B\delta(\alpha)$ then $\gamma(\beta)=_B\delta(\beta)$. It follows from (T2) and the transitivity of $=_B$. The inductive cases are straightforward. 
\end{proof}

In the proof of the soundness we  use the  following property to verify that $a\colon \vp$. \smallskip

(S) Let $\vp$ be $\forall x_1\ldots \forall x_k(\vp_1\imp\ldots \imp \vp_{l+1})$. Say $x_i$ has type $A_i$ and $\vp_i$ has type $B_i$.
\begin{itemize}
	\item Suppose $a\colon \vp$. If $n_i=_{A_i}m_i$ for $i=1,\ldots,k$, $b_i=_{B_i}c_i$ and $b_i \colon \vp_i(F_{n_1},\ldots , F_{n_k})$ for $i=1,\ldots,l$,  then $an_1\ldots n_kb_1\ldots b_l=_{B_{l+1}}am_1\ldots m_kc_1\ldots c_l$ and \[ an_1\ldots n_kb_1\ldots b_l\colon \vp_{l+1}(F_{n_1},\ldots F_{n_k}). \]
	\item Suppose $an_1\ldots n_kb_1\ldots b_{l-1}$ is defined for all $n_i,b_i$. Moreover, suppose that  if $n_i=_{A_i}m_i$ for $i=1,\ldots,k$, $b_i=_{B_i}c_i$ and $b_i\colon\colon \vp_i(F_{n_1},\ldots , F_{n_k})$ for $i=1,\ldots,l$,  then $an_1\ldots n_kb_1\ldots b_l=_{B_{l+1}}am_1\ldots m_kc_1\ldots c_l$ and \[ an_1\ldots n_kb_1\ldots b_l\colon \vp_{l+1}(F_{n_1},\ldots F_{n_k}). \]
	Then $a\colon \vp$.  
\end{itemize} 

Also consider  the following  instances of (S). \smallskip

(S1) Let $\vp$ be $\forall x_1\ldots \forall x_k\vp$. Say $x_i$ has type $A_i$ and $\vp$ has type $B$. Then $a\colon \vp$ iff for all $n_i,m_i$  such that $n_i=_{A_i}m_i$ for $i=1,\ldots, k$, 
we have that $an_1\ldots n_k=_Bam_1\ldots m_k$ and $an_1\ldots n_k\colon \vp(F_{n_1},\ldots,F_{n_k})$.\smallskip 

(S2) Let $\vp$ be $\vp_1\imp\ldots\imp \vp_l\imp \vp_{l+1}$. Say $\vp_i$ has type $B_i$.
\begin{itemize}	\item Suppose $a\colon \vp$, $b_i=_{B_i}c_i$ and $b_i\colon \vp_i$ for all $i=1,\ldots,l$. Then $ab_1\ldots b_l=_{B_{l+1}}ac_1\ldots c_l$ and $ab_1\ldots b_l\colon \vp_{l+1}$. 	\item Suppose that $ab_1\ldots b_{l-1}$ is defined for all $b_i$. Moreover, suppose that whenever $b_i=_{B_i}c_i$ and $b_i\colon \colon \vp_i$ for all $i=1,\ldots,l$, then $ab_1\ldots b_l=_{B_{l+1}}ac_1\ldots c_l$ and $ab_1\ldots b_l\colon \vp_{l+1}$. Then $a\colon \vp$.
\end{itemize}

\begin{theorem}[Soundness]
	Suppose $\vp$ is a formula of $\mcal L$ of type $B$ such that $\ehaw+\ac+\dc\vdash \vp$.  If the free variables of $\vp$ are among the distinct variables $x_1,\ldots,x_k$, then
	there exists an index $a$ such that $a\colon \forall x_1\ldots\forall x_k\vp$. 
	In particular, if $\vp$ is a sentence, then there is an index $a$ such that $a\colon \vp$. 
\end{theorem}
\begin{proof}
	The indices are mostly the same as in the proof of Theorem \ref{sound}. 
	Only the indices for axioms and rules involving $\lor$, that is, (5), (7), and (21),  are slightly different.  The bulk of the proof is to check that the indices have the right type. The verification is routine. 
	For the sake of completeness, we write down the details.\\

	Assume $x_i$ is of type $A_i$. For the sake of brevity, we write $A$ for $A_1,\ldots, A_k$ and  $x^A$ for $x_1,\ldots,x_k$. For the remainder of the proof we assume that the free variables of $\vp$ are among $x^A$. We thus write $a\colon \forall x^A\vp$ or simply $a\colon \vp$ for $a\colon \forall x_1\ldots\forall x_k\vp$.   We write, e.g., $n$ for $n_1\ldots n_k$ so that $an$ stands for $an_1\ldots n_k$ and $n=_Am$ stands for $ n_1=_{A_1}m_1, \ldots, n_k=_{A_k}m_k$. Similarly, we write $\vp(F_n)$ for $\vp(F_{n_1},\ldots,F_{n_k})$.  Let $A\to B$ denote $A_1\to A_2\to \ldots \to A_k\to B$.\\
	
	We repeatedly use (S) to show that $a\colon \vp$. \\

	

	(1) $a\colon \vp^B\imp\vp^B$, where $anb=b$.  
	
	It is easy to see that $a\in A\to B\to B$ is an index for $\vp\imp\vp$.\\
	
	(2) Modus ponens: if $c\colon \vp^C$ and $b\colon \vp^C\imp\psi^D$ then $a\colon \psi^D$. 
	
	Suppose that there are variables, say $x^B$, that occur free in $\vp$ but are not among  $x^A$. Then the free variables of $\vp$ and $\vp\imp\psi$ are among $x^A,x^B$. By induction, let  $c\in A\to B\to C$  be an index for $\varphi$ and $b\in A\to B\to C\to D$ be an index for $\vp\imp \psi$. Choose $a$  such that $an\simeq bn0^B(cn0^B)$, where $0^B\in B$ are fixed.  We claim that  $a\in A\to D$ is as desired.
	
	By (S1), let us check that if $n=_{A}m$ then  $an=_Dam$ and $an\colon \psi(F_n)$.  
	Clearly, $0^B=_B0^B$. By induction, 
	$cn0^B=_Ccm0^B$, $cn0^B\colon \vp(F_{n},F_{0^B})$, $bn0^B=_{C\to D}bm0^B$, and  $bn0^B\colon \vp(F_{n},F_{0^B})\imp\psi(F_{n})$. It thus follows that $bn0^B(cn0^B)=_Dbm0^B(cm0^B)$, that is $an=_Dam$, and $bn0^B(cn0^B)\colon \psi(F_n)$, that is, $an\colon\psi(F_n)$, as desired. \\
	
	(3) Syllogism: if $c\colon \vp^C\imp\psi^D$ and $b\colon \psi^D\imp\chi^E$ then $a\colon \vp^C\imp\chi^E$. 
	
	Suppose that there are variables, say $x^B$, that occur free in $\psi$ but are not among $x^A$.  By induction, let $c\in A\to B\to C\to D$ be an index for $\vp\imp\psi$ and $b\in A\to B\to D\to E$ be an index for $\psi\imp\chi$.  Choose $a$ such that  $and\simeq bn0^B(cn0^Bd)$.  We claim that $a\in A\to C\to E$ is as desired.  
	
	By (S), it is enough show that if $n=_{A}m$, $d=_Ce$ and $d\colon \vp(F_n)$ then $and=_Eame$ and $and\colon \chi(F_n)$. By induction, $cn0^Bd=_Dcm0^Be$ and $cn0^Bd\colon\psi(F_n)$. On the other hand, $bn0^B=_{D\to E}bm0^B$ and $bn0^B\colon (\psi\imp \chi)(F_n)$. It thus follows that $bn0^B(cn0^Bd)=_E bm0^B(cm0^Be)$, that is, $and=_Eame$, and  $bn0^B(cn0^Bd)\colon \chi(F_n)$, that is, $and\colon\chi(F_n)$. \\

	(4) $a_i\colon \psi_0\land\psi_1\imp\psi_i$, where $anb=(b)_i$. Say $\psi_i$ has type $B_i$.
	
	Let us show that $a_i\in A\to B_0\times B_1\to B_i$ is as desired. By (S), it suffices to show that if $n=_Am$, $b=_{B_0\times B_1}c$ and $b\colon (\psi_0\land\psi_1)(F_n)$ then $anb=_{B_i}amc$ and $anb\colon \psi_i(F_n)$. Easy. The conlusion follows immediately from the definitions of equality and realizability for product types.\\
	
	(5) $a_i\colon \psi_i\imp\psi_0\lor\psi_1$. Say $\psi_i$ has type $B_i$.
	
	Case $i=0$. Choose $a$ such that $anb=\langle 0,b,0^{B_1}\rangle$. Let us show that $a\in A\to B_0\to N\times B_0\times B_1$ is as desired. Suppose $n=_Am$, $b=_{B_0}c$ and $b\colon \psi_0(F_n)$. By (S), it suffices to show that $anb=_{N\times B_0\times B_1}amc$ and $anb\colon (\psi_0\lor\psi_1)(F_n)$. Easy. The conclusion follows from the definitions of equality for product types and  realizability for disjunctions. 
	
	The case $i=1$ is similar. Choose $a$ such that $anb=\langle 1,0^{B_0},b\rangle$.\\

	(6) If $b\colon \chi^B\imp\vp^C$ and $c\colon \chi^B\imp\psi^D$ then $a\colon \chi^B\imp(\vp^C\land\psi^D)$, where $and\simeq \langle bnd, cnd\rangle$. 
	
	First, note that the free variables of the premises are among $x^A$. By induction, let $c\in A\to B\to C$ and $b\in A\to B\to D$ be indices for $\chi\imp\vp$ and $\chi\imp\psi$ respectively. Let us show that $a\in A\to B\to C\times D$ is as desired. Since $an$ is defined for all $n$, it is enough to show by (S) that if $n=_Am$, $d=_Be$ and $d\colon \chi(F_n)$ then $and=_{C\times D}ame$ and $and\colon (\vp\land\psi)(F_n)$. This is straightforward. \\

	(7) If $b\colon \vp^B\imp\chi^D$ and $c\colon \psi^C\imp\chi^D$ then $a\colon \vp^B\lor\psi^C\imp\chi^D$, where \[ and\simeq \begin{cases}  bn(d)_1 &  \text{ if } (d)_0=0\\
	cn(d)_2 &   \text{ otherwise}
	\end{cases} \]
	We claim that $a\in A\to N\times B\times C\to D$  is as desired. Note that $an$ is defined for all $n$. By (S), it suffices to show that if $n=_Am$, $d=_{N\times B\times C}e$ and $d\colon (\vp\lor\psi)(F_n)$ then $and=_Dame$ and $and\colon \chi(F_n)$. Supose $(d)_0=0$. The other case is similar. Then $(d)_1=_B(e)_1$ and $(d)_1\colon \vp(F_n)$. By induction, $bn(d)_1=_Dbn(e)_1$, that is, $and=_Dame$,  and $bn(d)_1\colon \chi(F_n)$, that is, $and\colon \chi(F_n)$, as desired.  \\

	(8) If $b\colon \vp^B\land\psi^C\imp\chi^D$ then  $a\colon \vp^B\imp(\psi^C\imp\chi^D)$, 
	where $ance\simeq bn\langle c,e\rangle$. If $b\colon \vp^B\imp(\psi^C\imp\chi^D)$ then $a\colon  \vp^B\land\psi^C\imp\chi^D$, where $anc\simeq bn(c)_0(c)_1$.
	
	For the first rule, let us show that $a\in A\to B\to C\to D$ is as desired. Note tha $anc$ is defined for all $n,c$. By (S), it suffices to show that if $n=_Am$, $c=_Bd$ with $c\colon \vp(F_n)$, and $e=_Cf$ with $e\colon \psi(F_n)$, then $ance=_Damdf$ and $ance\colon \chi(F_n)$. Straightforward. 
	
	The second rule is similar. \\
	
	(9) $a\colon 0=S0\imp\vp^B$, where $anb=0$.
	
	Then $a\in A\to N\to B$ is an index for $0=S0\imp\vp$. Not that no $b$ realizes $0=S0$.  \\
	
	(10) If $b\colon \vp^C\imp\psi^D$ then $a\colon \vp^C\imp\forall x^B\psi^D$, for $x$ not free in $\vp$. 
	
	Suppose  $x^B$ is not among $x^A$.  Then the free variables of $\vp\imp\psi$ are among $x^A,x^B$. By induction, let $b\colon A\to B\to C\to D$ be an index for $\vp\imp \psi$. Choose $a$ such that $ance\simeq bnec$.  We claim that $a\in A\to C\to B\to D$ is an index for $\vp\imp\forall x\psi$. 
	
	Note that $anc$ is defined for all $n,c$. By (S), it is enough to show that is  $n=_Am$, $c=_C d$ and $c\colon \vp(F_n)$, then $anc=_{B\to D}amd$ and $anc\colon \forall x\psi(F_n,x)$. Suppose $e=_Bf$. We aim to show that $ance=_Damdf$ and $ance\colon \psi(F_n,F_e)$. By induction, $bnec=_Dbmfd$ and $bnec\colon \psi(F_n,F_e)$. By definition, $a$ is as required.
	
	If $x^B$ is among $x_1,\ldots,x_k$, say $x^B$ is $x_i$, then the free variables of $\vp\imp\psi$ are also among $x_1,\ldots, x_k$. Choose $a$ such that $ance\simeq bn_1\ldots n_{i-1}en_{i+1}\ldots n_kc$.
	Then $a$ is as desired. The proof is similar. \\

	(11) If $b\colon \vp^C\imp\psi^D$ then $a\colon \exists x^B\vp^C\imp\psi^D$, for $x$ not free in $\psi$.  
	
	Suppose that  $x^B$ is not among $x^A$.  Then the free variables of $\vp\imp\psi$ are among $x^A, x^B$. By induction, let  $b\in A\to B\to C\to D$ be an index for $\vp\imp\psi$.   Choose $a$ such that $anc \simeq bn(c)_0(c)_1$.   We claim that $a\in A\to B\times C\to D$ is as desired. 
	
	By (S), it is enough to show that if $n=_Am$, $c=_{B\times C}d$ and $c\colon \exists x^B\vp(F_n,x)$, then $anc=_Damd$ and $anc\colon \psi(F_n)$. By assumption, $(c)_0=_B(d)_0$, $(c)_1=_C(d)_1$ and $(c)_1\colon \vp(F_n,F_{(c)_0})$.  By induction, $bn(c)_0(c)_1=_Dbm(d)_0(d)_1$ and $bn(c)_0(c)_1\colon \psi(F_n)$. By definition, $a$ is as required.    
	
	If $x^B$ is among $x_1,\ldots,x_k$, say $x^B$ is $x_i$, then the free variables of $\vp\imp\psi$ are also among $x_1,\ldots, x_k$. Choose $a$ such that $anc\simeq bn_1,\ldots, n_{i-1},(c)_0,n_{i+1},\ldots, n_k(c)_1)$. Then $a$ is as desired. The proof is similar.  \\

	(12) $a\colon \forall x^B\vp^C\imp\vp^C(\alpha)$ and $a\colon \vp^C(\alpha)\imp\exists x^B\vp^C$, for any term $\alpha$ free for $x^B$ in $\vp$,  where  $anb\simeq b(dn)$ and   $anb\simeq \langle dn,b\rangle$ respectively, and $dn\simeq |\alpha(F_n)|$. The index $d$ exists by (V2). 
	
	
	Let us check that  $a\in A\to (B\to C)\to C$ is an index for the first axiom. By (S), it suffices to show that if $n=_Am$, $b=_{B\to C}c$ and $b\colon \forall x^B\vp(F_n,x)$, then $anb=_Camc$ and $anb\colon \vp(F_n,\alpha(F_n))$. By assumption, $F_n=_AF_m$ and hence by (T2) we have that $\alpha(F_n)=_B\alpha(F_m)$, that is, $dn=_Bdm$. By the assumption on $b$, it follows that $b(dn)_Cb(dm)$ and $b(dn)\colon \vp(F_n,F_{dn})$. As $F_{dn}=_B\alpha(F_n)$, by (R2) we have that $b(dn)\colon \vp(F_n,\alpha(F_n))$. By definition, $a$ is as required. 
	
	For the second axiom, let us check that $a\in A\to C\to B\times C$ is as desired. By (S), it suffices to show that if $n=_Am$, $b=_Cc$ and $b\colon \vp(F_n,\alpha(F_n))$, then $anb=_{B\times C}amc$ and $anb\colon \exists x^B\vp(F_n,x)$, that is, $dn=_Bdm$, $b=_C c$, $dn\in B$ and $b\colon \vp(F_n,F_{dn})$. It is clearly enough to prove that $dn=_Bdm$ and $b\colon \vp(F_n,F_{dn})$. As before, $\alpha(F_n)=_B\alpha(F_m)$, that is, $dn=_Bdm$. Also, $\alpha(F_n)=_BF_{dn}$, and hence by (R2) it follows that $b\colon \vp(F_n,F_{dn})$, as required.  \\

	(13) $a\colon\neg(0=Sx)$ and $a\colon Sx=Sy\imp x=y$, for  $x,y$ variables of type $N$, where $anb=0$.
	
	Easy. \\

	All the axioms from (14) to  (18) are given by atomic formulas (type $N$). Choose $a$ such that $an=0$ for all $n$. It is routine to check that $a\in A\to N$ is as required. \\ 
	
	(19) Induction: If $b\colon \vp^C(0)$ and $c\colon vp^C(x)\imp\vp^C(Sx)$ then $a\colon \vp^C(x)$, for $x$ of type $N$, where  $an0\simeq bn$ and $an(i+1)\simeq  cni(ani)$. 
	
	We may assume that $x^A$ is $x^B, x^N$.  By induction, let $b\in B\to C$ be an index for $\vp(0)$ and $c\in A\to C\to C$ be an index for $\vp(x)\imp\vp(x+1)$.  We claim that $a\in A\to C$ is an index for $\vp(x)$. 
	
	By (S1), it is enough to show that for all $n,m,i$, if  $n=_Bm$ then 
	$ani=_Cami$ and $ani\colon \vp(F_n,F_i)$. By induction on $i$. Let $i=0$. By induction, $bn=_Cbm$ and $bn\colon \vp(F_n,0)$. By (R2), $bn\colon \vp(F_n,F_0)$. Since $an0=bn$ and $am0=bm$, we are done. Suppose this is true for $i$ and let us prove the claim for $i+1$. By the assumption, $cni=_{C\to C}cmi$ and $cni\colon \vp(F_n,F_i)\imp \vp(F_n,SF_i)$. It follows from the induction hypothesis that $cni(ani)=_Ccmi(ami)$ and $cni(ani)\colon \vp(F_n,SF_{i})$. By (R2), $cni(ani)\colon \vp(F_n,F_{i+1})$. The claim follows from the definition of $a$. \\
	
	(20) $a\colon x=x$ for $x$ of type $A$, where $an=0$. Then $a\in A\to N$ is an index for $x=x$. 
	
	Easy.  \\
	
	(21)  $a\colon x=y\lor \neg (x=y)$ for $x$ of type $N$.
	
	We may assume that $x^A$ is $x,y,\ldots$ Choose $a$ such that $an=\langle 0,0,b\rangle$ if $n_1=n_2$, and $an=\langle 1,0,b\rangle$ otherwise, where $bc=0$ for all $c$.  Let us show that $a\in  A\to N\times N\times (N\to N)$ is as desired. Let $B=N\times N\times (N\to N)$. Note that $b\in N\to N$. 
	
	By (S), it is sufficient to show that if $n=_Am$ then $an=_Bam$ and $an\colon F_{n_1}=F_{n_2}\lor \neg(F_{n_1}=F_{n_2})$. Here we use the fact that equality between natural numbers is decidable.  Suppose $n_1=n_2$. As $n=_Am$, we have that $m_1=m_2$ and so $an=\langle 0,0,b\rangle=am$. It follows that $an=_Bam$. On the other hand, $0\colon F_{n_1}=F_{n_2}$, and therefore $an$ is as desired. Suppose $n_1\neq n_2$. As $n=_Am$, it follows that $m_1\neq m_2$ and so $an=\langle 1,0,b\rangle= am$. As before, $an=_Bam$. It remains to show that $an$ is a realizer, that is, $b\colon \neg(F_{n_1}=F_{n_2})$. Since $b\in N\to N$, it is sufficient to show that if $c\colon F_{n_1}=F_{n_2}$ then $bc\colon 0=S0$. Since $c\colon F_{n_1}=F_{n_2}$ iff $n_1=n_2$, the conclusion follows from ex falso.   \\
	
	(22)  Leibniz: $a\colon x=_B y\land \varphi^C(x)\imp \varphi^C (y)$, where $anb=(b)_1$. 
	
	We may assume that $x^A$ is $x,y,\ldots$ Let us show that $a\in A\to N\times C\to C$ is as desired.  Let  $n=_{A} m$,  $b=_{N\times C}c$ and $b\colon F_{n_1}=_BF_{n_2}\land \vp(F_n)$. By (S), we aim to show that $anb=_C amc$ and $anb\colon \vp(F_{n_2},F_{n_2},\ldots, F_{n_k})$.  The first part is straightforward. For the second part, we have by the assumption that $(b)_0\colon F_{n_1}=_BF_{n_2}$, that is, $n_1=_Bn_2$, and $(b)_1\colon  \vp(F_{n})$. Then $anb\colon \vp(F_n)$. The conclusion follows from  (R2). \\

	(23) Extensionality: $a\colon\forall z^B(xz=_Cyz)\imp x=_{B\to C}y$, where $anb=0$.
	
	We may assume that $x^A$ is $x,y,\ldots$ Let us check that  $a\in A\to (B\to N)\to N$ is as desired.  Suppose $n=_Am$,  $b=_{B\to N}c$ and $b\colon  \forall z^B(F_{n_1}z=_CF_{n_2}z)$. By (S), it is sufficient to show that $anb=amc$ and $anb\colon F_{n_1}=_{B\to C}F_{n_2}$, that is $n_1=_{B\to C}n_2$. The first part is trivial, since $anb=0=amc$. For the second part, note that $n_1,n_2\in B\to C$. By (E2),  it is sufficient to show  that if $d\in B$ then $n_1d=_Cn_2d$. By the assumption, if $d\in B$ then $bd\colon F_{n_1}F_d=_CF_{n_2}F_d$. It follows that $n_1d=_Cn_2d$, as desired. \\

	(24) Axiom of choice:
	$a\colon \forall x^B\exists y^{C_0}\varphi^{C_1}(x,y)\imp \exists z^{B\to C_0}\forall x^B\varphi^{C_1}(x,zx)$, where
	$anb=\pair{{a_0}nb}{a_1nb}$, and $a_inbd\simeq (bd)_i$.   We claim that \[ a\in A\to \underbrace{(B\to C_0\times C_1)}_D\to \underbrace{(B\to C_0)\times (B\to C_1)}_E \]
	is as desired.  Note that $anb$ is defined everywhere. Let $n=_Am$, $b=_{D}c$ and $b\colon \forall x\exists y\vp(F_n,x,y)$. By (S), it is sufficient to show  (i) $anb=_{E} amc$, that is, $a_inb=_{B\to C_i}a_imc$ for $i<2$, and (ii) $anb\colon \exists z\forall x\vp(F_n,x,zx)$. 	
	
	Let us prove  (i). Let $d=_Be$. From $b=_{D}c$ we have that $bd=_{C_0\times C_1}ce$, and so $(bd)_i=_{C_i}(ce)_i$. By the definition of $a_i$, we thus have $a_inbd=_{C_i}a_imce$. To prove (ii),  since $a_0nb\in B\to C_0$ by (i), it is enough to show that $a_1nb\colon \forall x\vp(F_n,x,F_{a_0nb}x)$. Similarly, since $a_1nb\in B\to C_1$ by (i),  it is enough to show that for all $d\in B$, $a_1nbd\colon \vp(F_n,F_d,F_{a_0nb}F_d)$.  As  $b\colon \forall x^B\exists y^{C_0}\varphi^{C_1}(F_n,x,y)$, we have that $(bd)_1\colon \vp(F_n,F_d,F_{(bd)_0})$. Now, $a_1nbd=(bd)_1$ and  $|F_{a_0nb}F_d|=a_0nbd=(bd)_0=|F_{(bd)_0}|$. It follows from (T2) and (R2) that $a_1nbd$ is as required. \\

	(25) Axiom of relativized dependent choice: \[ a\colon \forall x^B[\vp^C(x)\imp \exists y^B(\vp^C(y)\land \psi^D(x,y))]\imp \forall x^B[\vp^C (x)\imp \exists z^{N\to B}(z0=_Bx\land \forall v^N\psi^D(zv,z(Sv)))].\] 
	
	For simplicity, suppose that $x^A$ is empty. Let $abnd\simeq \langle fbnd, 0,gbnd\rangle$, where
	$fbndi\simeq (hbndi)_0$, $gbndi\simeq (hbnd(i+1))_{2}$, and $h$ is such that $hbnd0=\langle n, d,0\rangle$ and $hbnd(i+1)\simeq b(hbndi)_0(hbndi)_{1}$.
	We claim that 
	\[  a\in \underbrace{(B\to C\to B\times C\times D)}_{E} \to B\to C\to \underbrace{(N\to B)\times N\times (N\to D)}_{F} \]
	is as desired.   Note that $abnd$ is defined everywhere. By (S), it is sufficient to show that if $b=_E c$, $b\colon  \forall x[\vp(x)\imp \exists y(\vp(y)\land \psi(x,y))]$,
	$n=_Bm$, $d=_Ce$ and $d\colon \vp(F_n)$, then $abnd=_Facme$ and $abnd\colon  \exists z(z0=_BF_n\land \forall v\psi(zv,z(Sv)))$. Note that $(abnd)_1=0=(acme)_1$. It is thus sufficient to show that:
	\begin{enumerate}
		\item[(i)] $F_{fbnd}0=_BF_n$
		\item[(ii)] $fbnd=_{N\to B} fcme$  
		\item[(iii)] $gbnd=_{N\to D}gcme$ and $gbnd\colon \forall v\psi(F_{fbnd}v, F_{fbnd}(Sv))$
	\end{enumerate}
	Item (i) follows from (T2). In fact,  $|F_{fbnd}0|\simeq fbnd0$ and by definition  $fbnd0=n$.   It remains to show (ii) and (iii). We  claim that for all $i$:
	\begin{enumerate}
		\item[(a)] $fbndi=_B fcmei$
		\item[(b)]  $(hbndi)_{1}=_C(hcmei)_{1}$ 
		\item[(c)]  $(hbndi)_{1}\colon \vp(F_{fbndi})$ 
	\end{enumerate}
	By induction on $i$. Let $i=0$. We have that  $fbnd0=n$, $fcme0=m$, $(hbnd0)_{1}=d$ and $(hcme0)_{10}=e$. It follows from the assumptions $n=_Bm$, $d=_Ce$ and $d\colon \vp(F_n)$ that (a)--(c) hold true for $i=0$.  Suppose the claim is true for $i$. Let us prove the claim for $i+1$. By the induction hypothesis and the assumptions $b=_Ec$ and $b\colon \forall x^B[\vp(x)\imp \exists y(\vp(y)\land \psi(x,y))]$, it follows that   $b(fbndi)(hbndi)_{1}=_{B\times C\times D} c(fcmei)(hcmei)_{1}$ and 
	$b(fbndi)(hbndi)_1\colon \exists y(\vp(y)\land \psi(F_{fbndi},y))$. By the definition of $h$, we thus have
	$hbnd(i+1)=_{B\times C\times D}hcme(i+1)$, and so (a) and (b) hold true for $i+1$, and  $hbnd(i+1)\colon \exists y(\vp(y)\land \psi(F_{fbndi},y))$, and hence (c) holds true for $i+1$. This completes the proof of the claim.  
	
	Clearly, (ii) follows from (a).  Finally, let us prove (iii), that is, for all $i$,   $gbndi=_Dgcmei$ and $gbndi\colon \psi(F_{fbnd}F_i, F_{fbnd}(SF_i))$. Let $i$ be given. 
	As in the proof of the claim,  we have that $hbnd(i+1)=_{B\times C\times D}hcme(i+1)$ and  $hbnd(i+1)\colon \exists y(\vp(y)\land \psi(F_{fbndi},y))$. By the definition of $g$, we thus have that $gbndi=_Dgcmei$, as desired, and $gbndi\colon\psi(F_{fbndi},F_{fbnd(i+1)})$. By 
	(R2), it follows that  $gbndi\colon \psi(F_{fbnd}F_i, F_{fbnd}(SF_i))$. 
\end{proof}

We can formalize the above proof in $\ha$. We thus obtain:

\begin{theorem}
	Let $\vp$ be a sentence of $\mcal L$. If $\ehaw+\ac+\dc\vdash \varphi$ then there is an index $a$ such that $\ha\vdash a\colon \vp$.		
\end{theorem}


\subsection{Remarks}\label{why}
One can define a  Goodman version $p\fo a\colon \vp$ of this realizability notion. For instance, the clause for implication would be
\begin{itemize}
	\item $p\fo a\colon \vp^A\imp\psi^B$ iff $p\fo a\in A\to B$ and for all $q\supseteq p$ and for all $b$, if $q\fo b\colon \vp^A$ then  there is $r\supseteq q$ such that $r\fo ab\colon \psi^B$.
\end{itemize}
This interpretation is sound, that is,  for every sentence $\vp$ provable in $\ehaw+\ac+\dc$ and for every definable set $T$ of partial functions there is an index $a$ such that $\ha\vdash \forall p\in T\ p\fo a\colon \vp$. However, the self-realizability theorem fails in $\ha$. Let us see what goes wrong. 

Suppose we want to prove that for every first-order formula $\vp(x_1,\ldots,x_k)$ of type $B$ there exist a set $T$ of finite functions and an index $a$  such that:
\begin{enumerate}
	\item[(i)] For all $n$ and for all $p\in T$ there is $q\supseteq p$ such that $q\fo an\in B$. 
	\item[(ii)] If $\vp(\bar n)$ is true, then for all $p\in T$ there is $q\supseteq p$ such that $q\fo an \colon \vp(\bar n)$
	\item[(iii)] If $p\fo b\colon \vp(\bar n)$ then $\vp(\bar n)$ is true,
\end{enumerate}
where $n=n_1\ldots n_k$ and $\vp(\bar n)$ is $\vp(\bar n_1,\ldots, \bar n_k)$. Choose $T$ as in the proof of Theorem \ref{self}. The proof goes smoothly except for $\lor$ and $\exists$. In fact,  we need classical logic to prove $(i)$, and so we cannot formalize the proof  in $\ha$. Note however that the argument can be carried out in $\pa$. In more detail,  consider $\psi_j=\psi_0\lor\psi_1$, with $\psi_i$ of type $B_i$. Suppose $a_i$  works for $\psi_i$. Let 
\[   a^pn\simeq \begin{cases}
\langle 0,a_0^pn,0^{B_1}\rangle & \text{ if } p(j,n)=0 \\
\langle 1,0^{B_0},a_1^pn\rangle & \text{ if } p(j,n)=1 
\end{cases}\] 
Fix $n,p$. By induction we can find $q\supseteq p$ such that ${a_i}^qn$ is defined for every $i<2$. However we cannot prove that there is an extension of $q$ in $T$ which is defined on $j,n$. In fact, we should be able to decide whether $\psi_0(\bar n)\lor \psi_1(\bar n)$ is true. Similarly, if $\psi=\exists x\psi_0$ and $a$ is as in the proof of Theorem \ref{self}, then   we cannot prove that for all $n$ and $p$ there is $q\supseteq p$ such  that $a^qn$ is defined. In fact, as before, we should be able to decide whether $\exists x \psi_0(\bar n)$ is true.  

\section*{Note} After submitting the paper, the author was informed by Fernando Ferreira of a recent work by Benno van den Berg and Lotte van Slooten \cite{BS} on Goodman's theorem. Combining different ideas from previous proofs, in particular Lavalette \cite{lava}, they give another proof  of  Goodman's theorem and its extensional version. To do so, they define the system $\sf HAP$, a variant of Beeson's $\sf EON$ (Elementary theory of Operations and Numbers) and Troelstra's $\sf APP$, which is a conservative extension of $\ha$  based on the logic of partial terms introduced by Beeson \cite{Beeson85}. They then show that (1) if $\haw+\ac\vdash \vp$ then ${\sf HAP}\vdash \exists x(x\ \mathbf{r}\ \vp)$, where $x\ \mathbf{r}\ \vp$ is a realizability interpretation of formulas $\vp$ of $\haw$ into formulas of $\sf HAP$ which combines Kleene recursive realizability with  the model $\sf HRO$. They finally show that for a conservative extension ${\sf HAP}_\varepsilon$ of $\sf HAP$, obtained from $\sf HAP$ by adding \emph{Skolemization axioms} of the form $\exists y\vp(x,y)\imp f_\vp\cdot x\downarrow$ and $f_\vp\cdot x\imp \vp(x, f\cdot x)$ for every arithmetical formula of $\sf HAP$, one proves that (2) ${\sf HAP}_\varepsilon\vdash \vp\biimp \exists x(x\ \mathbf{r}\ \vp)$ for every arithmetical formula $\vp$.
Putting the pieces together one obtains for every arithmetical formula $\vp$
\[  \haw+\ac\vdash \vp \Rightarrow  {\sf HAP}\vdash \exists x(x\ \mathbf{r}\ \vp) \Rightarrow {\sf HAP}_\varepsilon \vdash \vp \Rightarrow \ha\vdash \vp \]
It is worth noticing that the conservativity proof of ${\sf HAP}_\varepsilon$ over $\sf HAP$ is based on a forcing interpretation $p\fo \vp$ of formulas $\vp$ of ${\sf HAP}_\varepsilon$ that resembles our forcing relation and that the proof of (2) is similar to the proof of the self-realizability theorem.  

For $\ehaw+\ac$ they define an extensional version $x=y\ \mathbf{e}\ \vp$ of Kleene recursive realizability combined with the model $\sf HEO$, and they prove corresponding (1) and (2). This notion resembles our $(a,b)\colon \vp$.

\subsection*{Acknowledgements}

The author thanks the anonymous referees who made very helpful comments and suggestions.
This research was supported by the Funda\c{c}\~{a}o para a Ci\^{e}ncia e a Tecnologia [UID/MAT/04561/2013] and Centro de Matem\'{a}tica,  Applica\c{c}\~{o}es Fundamentais e Investiga\c{c}\~{a}o Operacional of Universidade de Lisboa.

\bibliographystyle{plain}

\end{document}